\documentclass[12pt,leqno]{amsart}
\usepackage{amsmath,amsfonts,latexsym,graphicx,amssymb,url, color,cite}

\usepackage[pagebackref=true]{hyperref}
\usepackage{txfonts} 

\usepackage[hmargin=2.5 cm,vmargin=2.0 cm]{geometry}

\usepackage{graphicx}
\usepackage{color} 

\newcommand{\dist}{\text{dist}} 

\renewcommand{\div}{\mbox{div}\,}
\renewcommand{\div}{\mbox{div}\,}

\newcommand{\R}{{\mathbb R}} \newcommand{\K}{{\mathbb K}}

\newcommand{\G}{{\mathbb G}}
\newcommand{\Q}{{\mathbb Q}}

\newcommand{\N}{{\mathbb N}}

\newcommand{\e}{\varepsilon} 

\newcommand{\F}{\mathbf{F}}
\newcommand{\f}{\mathbf{f}}
\newcommand{\bzeta}{\mathbf{\zeta}}

\newcommand{\A}{\mathbf{A}}

\newcommand{\ba}{\mathbf{a}}

\newcommand{\M}{{\mathcal M}}

\def\longequals{\mathbin{=\kern-2pt=}}
\def\eqdef{\mathbin{\buildrel \rm def \over \longequals}}

\newtheorem{theorem}{Theorem}[section]
\newtheorem{corollary}[theorem]{Corollary}
\newtheorem{definition}[theorem]{Definition}
\newtheorem{remark}[theorem]{Remark}
\newtheorem{lemma}[theorem]{Lemma}
\newtheorem{proposition}[theorem]{Proposition}

\numberwithin{equation}{section}

\newcommand{\beq}{\begin{equation}}
\newcommand{\eeq}{\end{equation}}

\definecolor{darkred}{rgb}{.70,.12,.20}

\definecolor{darkgreen}{rgb}{.20,.52,.14}

\begin{document}
\title[Gradient estimates for  quasilinear elliptic equations]{Interior gradient estimates for  quasilinear elliptic equations
}


\author[T. Nguyen]{Truyen  Nguyen$^\ddag$} 
\address{$^\ddag$Department of Mathematics, University of Akron, 302 Buchtel Common, Akron, OH 44325--4002, U.S.A}
\email{tnguyen@uakron.edu}

\author[T. Phan]{Tuoc  Phan$^{\dag}$}
\address{$^{\dag\dag}$ Department of Mathematics, University of Tennessee, Knoxville, 227 Ayress Hall, 1403 Circle Drive, Knoxville, TN 37996, U.S.A. }
\email{phan@math.utk.edu}


\begin{abstract} 
We study quasilinear  elliptic equations of the form $\div \A(x,u,\nabla u) = \div \F $ in  bounded domains in $\R^n$, $n\geq 1$. The vector field $\A$ is 
allowed to be discontinuous in $x$, Lipschitz continuous in $u$ and its  growth in the gradient variable  is  like the $p$-Laplace operator with $1<p<\infty$. We establish interior 
$W^{1,q}$-estimates for locally bounded weak solutions to the  equations for every $q>p$, and we show that similar results also hold true in the setting of {\it Orlicz} spaces. Our regularity 
estimates extend results which are only known  for the case  $\A$ is independent of $u$ and they   complement  the well-known interior $C^{1,\alpha}$- estimates obtained by DiBenedetto \cite{D} and Tolksdorf \cite{T}
for general quasilinear elliptic equations.
\end{abstract}

\maketitle

\setcounter{equation}{0}

\section{Introduction}

We will investigate interior regularity for weak solutions to degenerate quasilinear elliptic equations of the form 
\begin{equation}\label{GE}
\div \A(x, u, \nabla u)=  \div \F  \quad \mbox{in}\quad \Omega,
\end{equation}
where $\Omega$ is a bounded domain in $\R^n$, $n\geq 1$. Without loss of generality we take $\Omega$ to be the Euclidean ball $B_6  := \{x\in \R^n:\, | x| <6\}$. 
Let  $\K\subset \R$ be an open interval and consider general vector field
\[
\A = \A(x,z,\xi) : B_6\times \overline\K\times \R^n \longrightarrow \R^n
\]
which is  a Carath\'eodory map, that is, $\A(x,z,\xi)$ is measurable in $x$ for every $(z,\xi)\in \overline \K \times \R^n$ and continuous in $(z,\xi)$ for a.e. $x\in B_6$. We assume that there exist constants  $\Lambda>0$ and  $1< p<\infty$ such that $\A$  satisfies the following  structural conditions for a.e. $x\in B_6$:
\begin{align}
&\big\langle  \A(x,z,\xi) -\A(x,z,\eta), \xi-\eta\big\rangle \geq \Lambda^{-1} \big(|\xi| +|\eta|)^{p-2} |\xi-\eta|^2\,\,\, \forall z\in \overline\K\mbox{ and } \forall \xi,\eta\in\R^n,\label{structural-reference-1}\\ 
& |\A(x,z,\xi)|  \leq \Lambda |\xi|^{p-1}\qquad\qquad\qquad\qquad\qquad\qquad\qquad\,\, \forall (z,\xi)\in \overline\K\times \R^n,\label{structural-reference-2}\\
& |\A(x,z_1,\xi)-\A(x,z_2,\xi)|  \leq \Lambda |\xi|^{p-1} |z_1 - z_2| \qquad\qquad\qquad \forall z_1, z_2\in \overline\K\mbox{ and } \forall \xi\in \R^n. \label{structural-reference-3}
\end{align}
 We want to emphasize  that  \eqref{structural-reference-1}--\eqref{structural-reference-3} are required to hold only for $z\in \overline \K$. This is useful   since in  some applications,  \eqref{structural-reference-1}--\eqref{structural-reference-3}  are satisfied only when $\K$ is a strict subset of $\R$ (see \cite{HNP1} for such an example where $\K=(0, M_0)$ for some constant $M_0>0$).

 The class of equations of the form \eqref{GE} with $\A$ satisfying \eqref{structural-reference-1}--\eqref{structural-reference-3} contains  the  well-known $p$-Laplace equations. The interior $C^{1,\alpha}$ regularity for  homogeneous $p$-Laplace equations was established by Ural\'tceva \cite{Ur}, Uhlenbeck \cite{Uh}, Evans \cite{E1} and Lewis \cite{Le}, while interior $W^{1,q}$-estimates for nonhomogeneous $p$-Laplace equations were obtained by Iwaniec \cite{I} and  DiBenedetto and Manfredi \cite{DM}. 
More generally,  \eqref{GE}   includes  equations of the type
  \begin{equation}\label{like-p-Laplacian}
\div \A(x,  \nabla u)=  \div \F  \quad \mbox{in}\quad \Omega
\end{equation}
whose $W^{1,q}$ regularity has been studied  by several authors when $\A$ is not necessarily continuous in the $x$ variable \cite{BW, BWZ, CP, Di, DM2, DM3, I, KZ,
 MPS, MP, P}. 

In this paper we study  general quasilinear equations 
\eqref{GE} when the principal parts also depend on the $z$ variable.  In the case $\A$ is  Lipschitz continuous in both $x$ and $z$ variables, the interior $C^{1,\alpha}$ regularity for locally bounded weak solutions to the corresponding homogeneous equations was established by DiBenedetto \cite{D} and Tolksdorf \cite{T} (see also \cite{Li} and the books \cite{GT,LU,MZ} for further  results). When $\A$ is  discontinuous in $x$, one does not expect  H\"older estimates for gradients of weak solutions and it is natural to search for $L^q$- estimates for the gradients instead. However, this type of  estimates for solutions to  \eqref{GE} is not well understood  even if $\F =0$. 
 Our main purpose of the current work  is to address this issue by establishing $W^{1,q}$-estimates for locally bounded weak solutions to the nonhomogeneous equation \eqref{GE} when $\A$ is not necessarily continuous in the $x$ variable and $ \F$ belongs to the $L^q$ space. We achieve this in  Theorem~\ref{main-result} 
 whose  particular consequence gives the following result:  

\begin{theorem}\label{simplest-main-result} Let $\K\subset \R$ be an open interval and $M_0>0$. Let $\A: B_6\times \overline\K\times \R^n \longrightarrow \R^n$ be a  Carath\'eodory map  such that $(z,\xi) \mapsto\A(x,z,\xi)$ is differentiable on $  \K\times (\mathbb{R}^n\setminus \{0\})$ for a.e. $x\in B_6$. Assume   that $\A$ satisfies 
\eqref{structural-reference-1}--\eqref{structural-reference-3} and the following   conditions for a.e. $x\in B_6$ and for all $z\in \overline{\K}$:
\begin{align*}
&\langle \partial_\xi \A(x,z,\xi) \eta, \eta\rangle \geq \Lambda^{-1} |\xi|^{p-2} |\eta|^2\, \quad \forall \xi\in \R^n\setminus \{0\} \mbox{ and }\forall \eta\in\R^n,\\ 
&   | \partial_\xi \A(x,z,\xi) |  \leq \Lambda |\xi|^{p-2}\qquad\qquad\quad \forall \xi \in \R^n\setminus \{0\}.
\end{align*}
Then for any $q>p$, there exists a constant  $\delta=\delta(p,q,n ,\Lambda,  \K, M_0)>0$  such that: if 
\begin{equation}\label{smallness-BMO}
\sup_{0<\rho\leq 3}\sup_{y\in B_1} \fint_{B_\rho(y)} \Big[
\sup_{z\in \overline\K}\sup_{\xi\neq 0}\frac{|\A(x,z,\xi) - \A_{B_\rho(y)}(z,\xi)|}{|\xi|^{p-1}}
\Big] \,dx \leq \delta,
\end{equation}
and $u$  is a weak solution of 
\[\div \A(x, u, \nabla u)=  \div \F\quad \mbox{in}\quad B_6
\]
satisfying $\|u\|_{L^\infty(B_5)}\leq M_0$, we have
\begin{equation*}
\|\nabla u\|_{L^q(B_1)} \leq C\left(
\|u\|_{L^p(B_6)} + \||\F|^{\frac{1}{p-1}}\|_{L^q(B_6)} \right).
\end{equation*}
Here $\A_{B_\rho(y)}(z,\xi) := \fint_{B_\rho(y)} \A(x,z,\xi)\, dx$ and  $C$ is a constant 
 depending only on  $q$, $p$, $n$,  $\Lambda$,  $\K$ and $M_0$.
\end{theorem}

Condition \eqref{smallness-BMO} means that the $BMO$ modulus of $\A$ in the $x$ variable is sufficiently small and hence it is automatically satisfied when $x\mapsto \A(x,z,\xi)$  is of vanishing mean oscillation. In particular, \eqref{smallness-BMO}  allows $\A$ to be discontinuous in $x$. We note that some smallness condition in $x$ for $\A$ is necessary since  it was known from Meyers' work  \cite{Me} that  in general  weak solutions to \eqref{like-p-Laplacian} do not possess interior $W^{1,q}$-estimates for every $q>p$ even in the linear case 
(i.e. $\A(x,  \nabla u) = \A(x)\nabla u$ and $p=2$).

$W^{1,q}$ theory   for equation  \eqref{like-p-Laplacian} was pioneered by Caffarelli and Peral. In \cite{CP}, these authors derived interior $W^{1,q}$-estimates for solutions to  \eqref{like-p-Laplacian} when $\A$ is sufficiently close in the $L^\infty$ sense to its average in the $x$ variable in every small scales. For the case $\A(x,\xi)= \langle \A(x)\xi, \xi\rangle^{\frac{p-2}{2}}\A(x) \xi$ with the matrix $\A(x)$ being uniformly elliptic and bounded, Kinnunen and Zhou \cite{KZ} obtained interior $W^{1,q}$-estimates when $\A(x)\in VMO$, i.e. $\A(x)$ is of vanishing mean oscillation. Recently,    Byun and Wang \cite{BW} (see also \cite{BWZ})
were able to obtain $W^{1,q}$-estimates for \eqref{like-p-Laplacian} under the assumption that the $BMO$ modulus of $\A$ in the $x$ variable is sufficiently small. Our obtained estimates in Theorem~\ref{simplest-main-result} are the same spirit as \cite{BW} but for general quasilinear elliptic equations of the form \eqref{GE}.

The proofs of $W^{1,q}$-estimates for solutions to  \eqref{like-p-Laplacian} in the above mentioned work 
use the perturbation technique from \cite{C,CC,CP}
and rely essentially on the central fact that equations of this type are invariant with respect to dilations and rescaling of domains. Unfortunately, this is no longer true for equations of the general form \eqref{GE} and this presents a serious obstacle
in deriving  $W^{1,q}$-estimates 
for their solutions. Our idea to handle this issue is to enlarge the class of equations under consideration in a suitable way
by considering the associated quasilinear elliptic  equations with two parameters (see equation \eqref{ME}). The class of these  equations is the smallest one that is invariant with respect to dilations and rescaling of domains and that contains equations of the form \eqref{GE}.
Given the invariant structure, a key step in our derivation  of $W^{1,q}$-estimates for the solution $u$ is to be able to approximate  $\nabla u$ 
by a good gradient in $L^p$ norm in a suitable sense (see
Corollary~\ref{cor:compare-gradient}). However, with the more general class of equations there arise new difficulties in  
this task as   we need to obtain the approximation uniformly with respect to the two parameters.
We achieve this in Lemma~\ref{lm:compare-solution} and Corollary~\ref{cor:compare-gradient}, and it is crucial that
the constants $\delta$ there  are independent of the two parameters $\lambda$ and $\theta$.  The main technical point of this paper is Lemma~\ref{lm:compare-solution}
which is a key point in our proof and is obtained through a delicate compactness argument. This kind of compactness arguments with parameters was first introduced in our recent paper 
  \cite{HNP1}
  where   parabolic equations whose  principle parts are linear in the gradient variable were considered. Here we extend further the argument  to take care the highly nonlinear structure in gradient of our equation \eqref{GE}.  Enlarging the class of equations to ensure the invariances  while still being able to obtain intermediate estimates uniformly with respect to the enlargement  is the main reason for our achievement and is  the novelty of this work. 

Our obtained  interior $W^{1,q}$-estimates in  Theorem~\ref{simplest-main-result} for general quasilinear elliptic equations  extend the corresponding estimates derived  
in \cite{CP, BW} for 
equation \eqref{like-p-Laplacian}. These estimates complement the celebrated interior $C^{1,\alpha}$-estimates by DiBenedetto \cite{D} and Tolksdorf \cite{T} for  \eqref{GE}. In fact,  Theorem~\ref{simplest-main-result} is  a particular consequence of our more general result established in Theorem~\ref{main-result}. It is worth  pointing out  that Theorem~\ref{main-result} is  even new when restricted to the simpler equation \eqref{like-p-Laplacian}. Indeed,   we only assume that  the distance from  $\A(x,\xi)$ to  a large set of "good" vector fields to be small while the previous work requires 
the distance from  $\A(x,\xi)$ to its average in the $x$ variable to be small.  More importantly, we identify  the properties of these good vector fields and are able to implement the general idea that  weak solutions to \eqref{GE}  possess   interior $W^{1,q}$- estimates for  any $q\in (p, \infty)$ provided that the equation is sufficiently close to a homogeneous  equation of similar form whose Dirichlet problem  has a unique weak solution admitting interior $W^{1,\infty}$-estimates.

 The method of our proofs in this paper  is quite robust and we illustrate   this  in Subsection~\ref{subsec:Orlicz}
by showing  that the interior estimates obtained in Theorem~\ref{main-result} still hold true in the setting of  {\it Orlicz} spaces (see Theorem~\ref{second-result} for the precise statement).
 We end the introduction by noting that quasilinear equations of general structures 
\eqref{structural-reference-1}--\eqref{structural-reference-3}
 arise in several applications and the availability of $W^{1,q}$-estimates for their solutions might be helpful for answering some open questions in these problems. We refer readers to \cite{HNP1} for such an application of $W^{1,q}$-estimates.

\section{Quasilinear elliptic equations of $p$-Laplacian type and main results}\label{sec:main-result}
Our goal is to derive interior $W^{1,q}$- estimates for weak solutions to
\begin{equation}\label{ME-1}
\div \A(x, u, \nabla u)=  \div \F \quad \mbox{in}\quad B_6
\end{equation}
 for any $q\in (p, \infty)$. We shall show that this is possible if   \eqref{ME-1} is close to a homogeneous  equation of similar form whose Dirichlet problem  has a unique weak solution admitting interior $W^{1,\infty}$-estimates. 
 For this purpose,  we introduce in the next subsection the  class   of reference equations together with a quantity  used to  measure   the closeness between two equations. In Subsection~\ref{sub:eqparameters},
 we explain the reasons for enlarging the class of equations under consideration.
\subsection{The class of good reference equations}
Let $\eta: \overline\K\times [0,\infty)\to \R$ be such that $\lim_{r\to 0^+}\eta(z,r)=\eta(z,  0)=0$ for each $z\in \overline\K$. Let  $\G_{B_3}(\eta)$ denote the class of all  vector fields $\ba: B_3\times \overline\K\times \R^n \longrightarrow \R^n$ satisfying  conditions \eqref{structural-reference-1}-- \eqref{structural-reference-3} for a.e. $x\in B_3$  together with the following additional  properties:
\begin{itemize}
\item[(\bf{H1})] For a.e. $x\in B_3$ and every $z\in \overline{\K}$, the map $\xi\mapsto \ba(x,z,\xi)$ is continuously differentiable away from the origin with
 \begin{equation*}\label{structural-reference-4}
 |\partial_\xi \ba(x,z,\xi)|\leq \Lambda |\xi|^{p-2} \quad \forall \xi\in\R^n\setminus \{0\}. 
 \end{equation*}
\item[(\bf{H2})] For every $z\in \overline\K$, we have
\begin{equation*}\label{osc-ref-eq}
\sup_{\xi\neq 0}\sup_{x_0: \, B_r(x_0)\subset B_3 } \fint_{B_r(x_0)} \frac{\big|\ba(x,z,\xi) - \ba_{B_r(x_0)}(z,\xi)\big|}{|\xi|^{p-1}} \,dx\leq \eta(z,r)\quad \mbox{for all $r>0$ small}.
\end{equation*}
\item[(\bf{H3})] For any $M>0$ and  $0<R\leq 1$, if $\bar v$ is a weak solution to $\div \ba(x, \bar v, \nabla \bar v)=0$ in $B_{3R}$ satisfying $\|\bar v\|_{L^\infty(B_{3 R})}\leq M$ then we have
\begin{equation*}\label{Lipschitz-ass} 
\|\nabla \bar v\|_{L^\infty(B_{\frac{5r}{6}})}^p\leq C(p,n,\eta, \Lambda, \K, M) \fint_{B_r}{|\nabla \bar v|^p\, dx}\quad \forall 0<r\leq 3R.
\end{equation*}
  \end{itemize}
By taking $\ba(x,z,\xi) =\ba(\xi)$, it is clear that  the class $\G_{B_3}(\eta)$ is nonempty. In fact, it contains a large number of vector fields as shown in our recent paper \cite{HNP2}. 
Ones also find in \cite{HNP2} that the class of  vector fields considered in  \cite{D, T}  to derive 
 interior $C^{1,\alpha}$-
estimates for the corresponding homogeneous equations belongs to $\G_{B_3}(\eta)$ for $\eta(z,r)\equiv \gamma_1 r$ with  $\gamma_1 $ being some positive constant.
\begin{definition}\label{def:rescaled-class} Let $y\in B_1$,  $0<\rho\leq 3$, and  $B_\rho(y) := \{x\in \R^n:\, | x-y| <\rho\}$. 
\begin{enumerate}
\item[(i)] We define
\begin{align*}
\G_{B_\rho(y)}(\eta)
:=\left\{\ba(x,z,\xi) :=\ba
'\big(\frac{3(x-y)}{\rho},z,\xi\big) \, \mbox{ for } \, (x,z,\xi)\in B_\rho(y)\times \overline\K\times \R^n \big| \,\, \ba'\in \G_{B_3}(\eta) \right\}.
\end{align*}
\item[(ii)] Let $\A: B_\rho(y)\times \overline\K\times \R^n \longrightarrow \R^n$ be a Carath\'eodory map. Define the distance
\begin{equation*}
\dist\Big(\A, \G_{B_\rho(y)}(\eta)\Big) :=
 \inf_{\ba \in \G_{B_\rho(y)}(\eta)}\fint_{B_\rho(y)} \Big[
\sup_{z\in \overline\K}\sup_{ \xi\neq 0}\frac{|\A(x,z,\xi) - \ba(x,z,\xi)|}{|\xi|^{p-1}}
\Big] \,dx. 
\end{equation*}
\end{enumerate}
\end{definition}

\begin{remark}\label{existen
ce+uniqueness} We note that under conditions \eqref{structural-reference-1}--\eqref{structural-reference-3}  for the vector field $\ba$,  the  homogeneous equation $\div \ba(x, v, \nabla v)=0$ in $B_3$ admits the comparison principle (see \cite[Theorem~1.2]{CMPo}). This together with the classical existence result due to Leray and Lions \cite{LL,L} (see also \cite[Theorem~2.8]{CPo}) ensures that: for any $u\in W^{1,p}(B_3)$ with $u(x)\in \overline
\K\,$  for a.e. $x\in B_3$, the Dirichlet problem 
\begin{equation*}
\left \{
\begin{array}{lcll}
\div \ba(x, v,\nabla v) &=&   0 \quad &\text{in}\quad B_3, \\
v & =& u\quad &\text{on}\quad \partial B_3
\end{array}\right.
\end{equation*}
has a unique weak solution $v\in W^{1,p}(B_3)$ satisfying $v(x)\in \overline
\K\,$  for a.e. $x\in B_3$. 
\end{remark}

\subsection{Quasilinear  equations with two parameters}
\label{sub:eqparameters}
Let us consider a function  $u\in W_{loc}^{1,p}(B_{r R})$  such that $u(y)\in \overline\K$ for a.e. $y\in B_{r R}$ and
$u$ satisfies 
\begin{equation*}
\div \A(y,u,\nabla u) =  \div \F\quad \text{in}\quad B_{r R }
\end{equation*}
in the  sense of distribution. Then the rescaled  function 
\begin{equation}\label{scaled-solutions}
 v(x) := \frac{u(r x)}{\mu\,  r} \quad \mbox{for}\quad r, \, \mu>0
 \end{equation}
has the properties: $v(x)\in \frac{1}{\mu r}\overline\K$ for a.e. $x\in B_R$ and $v$ solves the equation
\begin{equation*}
\div \A_{\mu, r}(x,\mu r v,\nabla v) =  \div \F_{\mu, r}\quad \text{in}\quad B_R
\end{equation*}
in the distributional sense. Here,
\begin{equation*} 
\A_{\mu, r}(x,z,\xi) := \frac{\A(rx,  z, \mu \xi)}{\mu^{p-1}} \quad\mbox{ and }\quad  \F_{\mu, r}(x) := \frac{\F(r x)}{\mu^{p-1}}. 
 \end{equation*} 
 It is clear that if  $\A: B_{r R}\times \overline\K\times \R^n \longrightarrow \R^n$ satisfies  conditions \eqref{structural-reference-1}--\eqref{structural-reference-3}, then 
the rescaled vector field $ \A_{\mu, r}: B_{R}\times \overline\K\times \R^n \longrightarrow \R^n$ also satisfies the structural conditions 
with the same constants. 

The above observation  shows that  equations of type \eqref{ME-1} are not invariant with respect to the standard scalings 
\eqref{scaled-solutions}. This presents a serious obstacle in obtaining $W^{1,q}$-estimates for their solutions as they do not generate enough  estimates 
to carry out the proof by using existing methods.  Our idea is to enlarge the class by considering associated quasilinear equations with two parameters 
\begin{equation}\label{ME}
\div \Big[ \frac{\A(x, \lambda \theta u, \lambda \nabla u)}{\lambda^{p-1}}\Big]=  \div \F \quad \mbox{in }\, B_6
\end{equation}
 with $\lambda,\, \theta>0$. The class of these equations is the smallest one that is invariant with respect to the transformations \eqref{scaled-solutions}
and that contains equations of type \eqref{ME-1}. Indeed,  if $u$ solves \eqref{ME} and $v$ is given by  \eqref{scaled-solutions}, then $v$ 
 satisfies an equation of similar form, namely,  $
\div \big[ \frac{\A'(y, \lambda' \theta' v, \lambda' \nabla v)}{\lambda'^{p-1}}\big]=  \div \F'$ in  $B_{\frac6r}$
with $\A'(y, z, \xi) := \A(r y, z, \xi)$, $\F'(y) := \F(r y )/\mu^{p-1}$, $\lambda' := \mu \lambda$  and $\theta':=r\theta$.

Let us give the precise definition of weak solutions that is used throughout the paper.
\begin{definition}
\label{weak-sol}
Let $\F\in L^{\frac{p}{p-1}}(B_6;\R^n)$. A function $u\in W^{1,p}_{\text{loc}}(B_6)$ is called a weak solution  of \eqref{ME} if  $u(x)\in \frac{1}{\lambda \theta}\overline\K \,$ for a.e. $x\in B_6$ and 
\begin{equation*}
\int_{B_6} \Big\langle\frac{ \A (x,\lambda \theta u,\lambda \nabla u)}{\lambda^{p-1}}, \nabla \varphi \Big\rangle\, dx =\int_{B_6} \langle \F, \nabla \varphi\rangle \, dx\qquad \forall \varphi \in W_0^{1,p}(B_6).
 \end{equation*}

\end{definition}

Our main result on the interior regularity is the following theorem:
\begin{theorem}\label{main-result} Assume that  $\A$ satisfies   
\eqref{structural-reference-1}--\eqref{structural-reference-3}, and $M_0>0$. 
For any $q>p$, there exists a constant $\delta=\delta(p,q,n ,\Lambda, \eta, \K, M_0)>0$  such that: if $\lambda>0$, $0<\theta \leq 1$,
\begin{equation*}
\sup_{0<\rho\leq 3}\sup_{y\in B_1} \dist \Big(\A, \G_{B_\rho(y)}(\eta)\Big) \leq \delta,
\end{equation*}
and $u$  is a weak solution of \eqref{ME}
satisfying $\|u\|_{L^\infty(B_5)}\leq\frac{M_0}{\lambda \theta}$, then
\begin{equation}\label{main-estimate}
\|\nabla u\|_{L^q(B_1)} \leq C(p,q,n,\Lambda, \eta, \K, M_0)\left(
\|u\|_{L^p(B_6)} + \||\F|^{\frac{1}{p-1}}\|_{L^q(B_6)}\right).
\end{equation}
\end{theorem}

By taking $\lambda =\theta =1$ in Theorem~\ref{main-result}, we then obtain  $W^{1,q}$-estimates for weak  solutions to original equation
 \eqref{ME-1}. Another observation is that  any function $\f\in L^p(B_6)$ can be written in the form
 $\,\f =\div \nabla \psi$, where $\psi \in W_0^{1,2}(B_6)$ is the weak solution to the Dirichlet problem 
 \begin{equation*}
\left \{
\begin{array}{lcll}
\Delta \psi &=&   \f \quad &\text{in}\quad B_6, \\
\psi & =& 0\quad &\text{on}\quad \partial B_6.
\end{array}\right.
\end{equation*}
Moreover by the standard estimate using Riesz
potential (see  \cite[page 195]{MP} for an explanation), we have when  $1<l<n$ that  
 \[
 \|\nabla \psi\|_{L^{\frac{nl}{n-l}}(B_6)}\leq C(n,l) \, \|\f\|_{L^l(B_6)}.
 \]
 These facts together with Theorem~\ref{main-result} yield:
\begin{corollary} Assume that  $\A$ satisfies 
\eqref{structural-reference-1}--\eqref{structural-reference-3}, and $M_0>0$. 
For any $\max{\big\{1,\frac{np}{np+p-n}\big\}}<l<n$, there exists a constant $\delta=\delta(p,l,n ,\Lambda, \eta, \K, M_0)>0$  such that: if $\lambda>0$, $0<\theta \leq 1$,
\begin{equation*}
\sup_{0<\rho\leq 3}\sup_{y\in B_1} \dist \Big(\A, \G_{B_\rho(y)}(\eta)\Big) \leq \delta,
\end{equation*}
and $u$  is a weak solution of 
\begin{equation*}
\div \Big[ \frac{\A(x, \lambda \theta u, \lambda \nabla u)}{\lambda^{p-1}}\Big]=  \div \F +\f \quad \mbox{in}\quad B_6
\end{equation*}
satisfying $\|u\|_{L^\infty(B_5)}\leq\frac{M_0}{\lambda \theta}$, then
\begin{equation*}
\|\nabla u\|_{L^{\frac{nl(p-1)}{n-l}}(B_1)}^{p-1} \leq C(p,l,n,\Lambda, \eta, \K, M_0)\left(
\|u\|_{L^p(B_6)}^{p-1} + \|\F\|_{L^{\frac{nl}{n-l}}(B_6)} + \|\f\|_{L^l(B_6)}\right).
\end{equation*}
\end{corollary}

\section{Some elementary estimates}
In this section we derive  some  elementary estimates which will be used later. For the next lemma  we only consider the case $1<p<2$ since \eqref{structural-reference-1}  obviously yields a  better estimate  when $p\geq 2$.
\begin{lemma}\label{simple-est} Let $U\subset\R^n$ be a bounded open set.
 Assume that $\A:U\times \overline\K\times \R^n\longrightarrow \R^n$ satisfies  \eqref{structural-reference-1} for a.e. $x$ in $U$ and for some $1<p<2$. Then for any functions $u,v \in W^{1,p}(U)$ and any nonnegative function $\phi\in C(\overline{U})$, we have
 \begin{align}\label{s1}
  (1-\tau)\int_{U} |\nabla u -\nabla v|^p\phi \, dx 
 &\leq \tau \int_{U} |\nabla u |^p\phi\, dx\\
  &+ \Lambda 2^{2-p}\frac{p}{2}\Big(\frac{2-p}{2^{2-p}\tau}\Big)^{\frac{2-p}{p}} \int_{U}{ \langle \A(x, u, \nabla u) -\A(x,u,\nabla v), \nabla u - \nabla v \rangle \phi\, dx }\nonumber
   \end{align}
   for every $\tau>0$.
\end{lemma}
\begin{proof} Since $|\xi|+|\eta|\leq 2(|\xi| +|\xi-\eta|)$ and $1<p<2$, we have from \eqref{structural-reference-1} that
\begin{equation}\label{equiv-structural-1}
\langle  \A(x,z,\xi) -\A(x,z,\eta), \xi-\eta\rangle \geq \Lambda^{-1} 2^{p-2} \big(|\xi|+ |\xi-\eta|)^{p-2} |\xi-\eta|^2\quad \forall \xi,\eta\in\R^n.
\end{equation}
 Using Young's inequality,  the assumption $1<p<2$ and  \eqref{equiv-structural-1}, we then obtain
 \begin{align*}
 &\int_{U} |\nabla u -\nabla v|^p \phi\, dx 
 =\int_{U} \big[\big( |\nabla u| + |\nabla u-\nabla v| \big)\phi^{\frac{1}{p}}\big]^{\frac{p(2-p)}{2}} \big[\big(  |\nabla u| + |\nabla u-\nabla v|\big)^{\frac{p(p-2)}{2}}  |\nabla u -\nabla  v|^p\phi^{\frac{p}{2}} \big] \, dx \\
 &\leq \frac{\tau}{2^{p-1}} \int_{U} \big(  |\nabla u| + |\nabla u-\nabla v|\big)^p \phi\, dx 
 + \frac{p}{2}\Big(\frac{2-p}{2^{2-p}\tau}\Big)^{\frac{2-p}{p}} \int_{U} \big(  |\nabla u| + |\nabla u-\nabla v|\big)^{p-2}  |\nabla u -\nabla  v|^2 \phi\, dx \\
 &\leq \tau \int_{U} |\nabla u |^p\phi\, dx  + \tau \int_{U} |\nabla u-\nabla v|^p \phi\, dx\\
 &\quad + \Lambda 2^{2-p}\frac{p}{2}\Big(\frac{2-p}{2^{2-p}\tau}\Big)^{\frac{2-p}{p}} \int_{U}{ \langle \A(x, u, \nabla u) -\A(x,u,\nabla v), \nabla u - \nabla v \rangle \phi \, dx }.
   \end{align*}
   This gives the lemma as desired.
\end{proof}


The next two results are about basic $L^p$-estimates for gradients of weak solutions.
\begin{proposition}\label{pro:gradient-est} Assume that $\A:B_3\times \overline\K\times \R^n\longrightarrow \R^n$ satisfies \eqref{structural-reference-1} and \eqref{structural-reference-2}. Let  $w\in  W^{1,p}(B_3)$  be a weak solution  of  
\begin{equation*}
\left \{
\begin{array}{lcll}
\div \A(x, w,\nabla w) &=&   \div \F \quad &\text{in}\quad B_3, \\
w & =& \varphi\quad &\text{on}\quad \partial B_3,
\end{array}\right.
\end{equation*}
where $\varphi\in  W^{1,p}(B_3)$.   Then
 \begin{equation*}
\int_{B_3}{ |\nabla w|^p \,dx }
\leq C(p,n,\Lambda)\, \Big(\int_{B_3}{|\nabla \varphi|^p\,  dx } +\int_{B_3}{|\F|^{\frac{p}{p-1}} \, dx} \Big).
 \end{equation*}
\end{proposition}
\begin{proof}
 By using $w-\varphi$ as a test function, we get
 \begin{equation*}
  \int_{B_3}{\langle \A(x, w,\nabla w), \nabla w -\nabla \varphi \rangle \, dx}=\int_{B_3}{\langle \F, \nabla w -\nabla \varphi \rangle \, dx}
  \end{equation*}
which can be rewritten as
 \begin{align*}
   \int_{B_3}{\langle \A(x, w,\nabla w) -\A(x, w,0), \nabla w  \rangle  dx}
   =\int_{B_3}{\langle \A(x, w,\nabla w), \nabla \varphi \rangle  dx}
   +\int_{B_3}{\langle \F, \nabla w  \rangle  dx}
   -\int_{B_3}{\langle \F, \nabla \varphi \rangle  dx}.
 \end{align*}
On the other hand, it follows from  \eqref{structural-reference-1} that
\begin{align*}
\Lambda^{-1} \int_{B_3}{|\nabla w|^p\, dx} \leq  \int_{B_3}{\langle \A(x, w,\nabla w) -\A(x, w,0), \nabla w  \rangle \, dx}.
 \end{align*}
 Therefore, we obtain 
 \begin{align*}
  \Lambda^{-1} \int_{B_3}{|\nabla w|^p\, dx}
\leq  \Lambda \int_{B_3}{ |\nabla w|^{p-1}|\nabla \varphi|\,dx}
  +\int_{B_3}{|\F| |\nabla w|\, dx}
  +\int_{B_3}{|\F| |\nabla \varphi|\, dx}.
 \end{align*}
 From this and by applying Young's inequality, we deduce the conclusion of the lemma. 
\end{proof}

\begin{lemma}\label{W^{1,p}-est} Assume that $\A:B_4\times \overline\K\times \R^n\longrightarrow \R^n$ satisfies \eqref{structural-reference-1} and \eqref{structural-reference-2}. Let  $u\in  W_{loc}^{1,p}(B_4)$  be a weak solution  of  
\begin{equation}\label{EQ-B4}
\div \A(x,  u,\nabla u) =   \div \F  \quad \text{in}\quad B_4.
\end{equation}
 Then 
 \begin{equation*}
\int_{B_3}{ |\nabla u|^p \,dx }
\leq C(p,n,\Lambda) \, \Big(\int_{B_4}{|u|^p\,  dx } +\int_{B_4}{|\F|^{\frac{p}{p-1}}\, dx} \Big). 
 \end{equation*}
 \end{lemma}
\begin{proof}
Let  $\varphi\in C_0^\infty(B_4)$ be the standard nonnegative cut-off function which is $1$ 
on $B_3$.  Then, by 
multiplying the equation by $\varphi^p u$ 
and using  integration by parts
  we get
\begin{align*}
  \int_{B_4}{\langle \A(x, u,\nabla u)-\A(x, u, 0), \nabla u  \rangle \varphi^p \, dx}
 & =-p \int_{B_4}{\langle \A(x, u,\nabla u), \nabla \varphi  \rangle \varphi^{p-1} u \, dx}\\
  &\quad + \int_{B_4} \langle \F, \nabla u\rangle \varphi^p\, dx +
  p\int_{B_4} \langle \F, \nabla \varphi\rangle \varphi^{p-1} u\, dx.
 \end{align*}
 Therefore, it follows from  \eqref{structural-reference-1} and \eqref{structural-reference-2} that
 \begin{align*}
  \Lambda^{-1} \int_{B_4}{|\nabla u|^p \varphi^p \, dx} 
  \leq  p \Lambda \int_{B_4}{ |\nabla u|^{p-1}  |\nabla \varphi |   \varphi^{p-1}  |u|  \, dx}
  + \int_{B_4} |\F| | \nabla u| \varphi^p\, dx +
  p\int_{B_4}  |\F| |\nabla \varphi| \varphi^{p-1} |u|\, dx.
 \end{align*}
This together with Young's inequality yields 
   the lemma. 
\end{proof}
We end the section with a result giving a bound on the  $L^p$-norm of the difference between two gradients of weak solutions.
\begin{lemma}\label{lm:bound-solution}
Assume $\A:B_4\times \overline\K\times \R^n\longrightarrow \R^n$ and  $\ba:B_3\times \overline\K\times \R^n\longrightarrow \R^n$ satisfy \eqref{structural-reference-1} and \eqref{structural-reference-2}. Let $u\in  W_{loc}^{1,p}(B_4)$
 be a weak solution  of  \eqref{EQ-B4}
and $v\in  W^{1,p}(B_3)$   be a weak solution of
\begin{equation*}
\left \{
\begin{array}{lcll}
\div \ba(x, v,\nabla v) &=&  0 \quad &\text{in}\quad B_3, \\
v & =& u\quad &\text{on}\quad \partial B_3.
\end{array}\right.
\end{equation*}
Then 
\begin{equation}\label{difference-gradients}
 \int_{B_3}{|\nabla u -\nabla v|^p\, dx}\leq C(p,n,\Lambda)\int_{B_4} \big( |u|^p + |\F|^{\frac{p}{p-1}} \big) \, dx.
\end{equation} Moreover,
\begin{equation}\label{v-L^p-est}
\int_{B_3}{|v|^p\, dx}\leq C(p,n,\Lambda)\int_{B_4} \big( |u|^p + |\F|^{\frac{p}{p-1}} \big) \, dx.
\end{equation}
\end{lemma}
\begin{proof}
By
using $ u-v$ as a test function in the equations for $u$ and $v$, we get
\begin{align*}
-\int_{B_3}{ \langle  \ba(x, v,\nabla v), \nabla u - \nabla v\rangle   dx }
=-\int_{B_3}{ \langle  \A(x,  u, \nabla u) , \nabla u - \nabla v\rangle \, dx } +\int_{B_3}{\langle \F, \nabla u - \nabla v\rangle  dx}.
\end{align*}
This gives
\begin{align*}
 J &:=\int_{B_3}{ \langle  \ba(x, v , \nabla u) -\ba(x, v,\nabla v), \nabla u - \nabla v \rangle \, dx }\\
 &=\int_{B_3}{ \langle  \ba(x,  v, \nabla u) - \A(x, u,\nabla u), \nabla u - \nabla v \rangle \, dx}
 +\int_{B_3}{\langle \F, \nabla u - \nabla v\rangle \, dx}.
\end{align*}
It follows  from this and  \eqref{structural-reference-2} that
\begin{align*}
J\leq 2 \Lambda \int_{B_3}{ |\nabla u|^{p-1} | \nabla u - \nabla v|  \, dx}
+\int_{B_3}{|\F| |\nabla u - \nabla v| \, dx}.\nonumber
\end{align*}
Moreover,   Lemma~\ref{simple-est} and \eqref{structural-reference-1} imply that
\[
c \int_{B_3} |\nabla u -\nabla  v|^p \, dx 
- c^{-1}\int_{B_3} |\nabla u|^p \, dx
 \leq J.
\]
Therefore,  we conclude that
\begin{align*}
c \int_{B_3} |\nabla u -\nabla  v|^p \, dx 
\leq  c^{-1}\int_{B_3}  |\nabla u|^p \, dx +  2\Lambda \int_{B_3}{ |\nabla u|^{p-1} | \nabla u - \nabla v|   \, dx}
+ \int_{B_3}{|\F| |\nabla u - \nabla v|  \, dx}.
\end{align*}
We infer from this and  Young's inequality  that 
\begin{equation*}
 \int_{B_3}{|\nabla u -\nabla v|^p\, dx}\leq C(p,n,\Lambda)\int_{B_3} \big( |\nabla u|^p + |\F|^{\frac{p}{p-1}} \big) \, dx.
\end{equation*}
This together with Lemma~\ref{W^{1,p}-est}  yields \eqref{difference-gradients}.
On the other hand,   \eqref{v-L^p-est} is a consequence of \eqref{difference-gradients} and the estimate
\begin{align*}
 \int_{B_3} |v|^p \, dx 
 \leq 2^{p-1}\Big[\int_{B_3}|u|^p \, dx +\int_{B_3}|u-v|^p \, dx \Big]\leq C(p,n)\Big[\int_{B_3}|u|^p \, dx + \int_{B_3}|\nabla u-\nabla v|^p \, dx\Big].
\end{align*}
 \end{proof}

\section{Approximating solutions }\label{approximation-solution}
The goal of this section is to  prove a result allowing us to  compare solutions originating from two different
equations.

\subsection{Strong compactness of the class $\G$ of vector fields}

In this subsection we give some elementary arguments showing that 
the class of  vector fields $\G_{B_3}( \eta)$ is relatively compact with respect to the pointwise  convergence. Let us first recall the sequential Bocce criterion in \cite{BGJ}.

\begin{definition}\label{dfn:Bocce}
We say that a sequence $\{f_k\}$ in $L^1(B_3;\R^n)$ satisfies the sequential
Bocce criterion if for each subsequence $\{f_{k_j}\}$ of $\{f_k\}$, each $\epsilon>0$ and each measurable set $E\subset B_3$ with $|E|>0$,
there exists a  measurable set $A\subset E$ with $|A|>0$ such that 
\begin{equation}\label{Bocce}
\liminf_{j\to \infty} \fint_{A} |f_{k_j}(x) - (f_{k_j})_{A}| \, dx <\epsilon. 
\end{equation}
\end{definition}
The following result is a special case of \cite[Theorem 2.3]{BGJ}.

\begin{theorem}\label{thm:L^1-conv}{(Theorem 2.3 in \cite{BGJ})}
Let $\{f_k\}$ be a sequence in $L^1(B_3;\R^n)$.  Then $\{f_k\}$ converges strongly to $f$ in $L^1(B_3;\R^n)$ if and only if
\begin{itemize}
\item[(1)] $\{f_k\}$ converges weakly to $f$ in $L^1(B_3;\R^n)$.
\item[(2)]  $\{f_k\}$ satisfies the sequential Bocce criterion. 
\end{itemize}
\end{theorem}
An application of Theorem~\ref{thm:L^1-conv} gives: 

\begin{lemma}\label{lm:sufficient-condition} Let $\eta: [0,\infty)\to \R$ be a function satisfying $\lim_{r\to 0^+}{\eta(r)}=0$.
Suppose $\{f_k\}$ converges weakly to $f$ in $L^1(B_3;\R^n)$, and
\[
\sup_{k} \sup_{x_0:\, B_r(x_0)\subset B_3 }\fint_{B_r(x_0)} |f_k(x) - (f_k)_{B_r(x_0)}|\, dx\leq \eta(r)
\]
for all $r>0$ sufficiently small. Then
$\{f_k\}$ converges strongly to $f$ in $L^1(B_3;\R^n)$. 
\end{lemma}
 \begin{proof}
 By Theorem~\ref{thm:L^1-conv}, it is enough to check that $\{f_k\}$ satisfies the sequential Bocce criterion. For this, let $\epsilon>0$  and let $E\subset B_3$ be a measurable set with $|E|>0$. Then by the Lebesgue differentiation theorem, there exists $x_0\in E$ such that 
 \begin{equation}\label{density-point}
 \lim_{r\to 0^+}\frac{|E\cap B_{r}(x_0)|}{|B_r(x_0)|}= 1\quad \mbox{and}\quad  \lim_{r\to 0^+}\fint_{B_r(x_0)} f=\lim_{r\to 0^+}\fint_{B_r(x_0)} f\chi_E =f(x_0).
 \end{equation}
For all $r>0$ small, we have with $A_r := E\cap B_{r}(x_0)$ that
 \begin{align*}
 \fint_{A_r} |f_{k}(x) - (f_{k})_{A_r}| \, dx
 &\leq \fint_{A_r} |f_{k}(x) - (f_k)_{B_r(x_0)}| \, dx +\Big|\fint_{B_r(x_0)} f_k - \fint_{A_r} f_k \Big|\\
 &\leq \frac{|B_r(x_0)|}{|A_r|} \eta(r) +\Big|\fint_{B_r(x_0)} f_k - \fint_{A_r} f_k \Big|\quad \forall k.
 \end{align*}
 It follows by taking $k\to \infty$ and using the weak convergence of  $\{f_k\}$ to $f$ that
 \begin{align*}
 \limsup_{k\to\infty}\fint_{A_r} |f_{k}(x) - (f_{k})_{A_r}| \, dx
 &\leq  \frac{|B_r(x_0)|}{|A_r|} \eta(r) +\Big|\fint_{B_r(x_0)} f - \fint_{A_r} f \Big|\\
 &= \frac{|B_r(x_0)|}{|A_r|} \eta(r) +\Big|\fint_{B_r(x_0)} f - \frac{|B_r(x_0)|}{|A_r|}\fint_{B_r(x_0)} f\chi_E \Big|\quad \forall r>0\mbox{ small}.
 \end{align*}
 Thanks to \eqref{density-point} and the assumption $\lim_{r\to 0^+}{\eta(r)}=0$, we can choose $r>0$ sufficiently small  such that the above  right hand side is less than $\epsilon$. Thus $\{f_k\}$ satisfies the sequential Bocce criterion and the proof is complete.
 \end{proof}
 
The strong compactness of  $\G_{B_3}( \eta)$ is given by the next result. This  technical lemma will be used in Subsection~\ref{sub:approximation-solu}.

\begin{lemma}\label{coeff-compactness}
For each positive integer  $k$, let $\ba_k: B_3\times \overline\K\times \R^n \to \R^n$ be a vector field satisfying conditions 
\eqref{structural-reference-2}--\eqref{structural-reference-3} and  ({\bf H1})--({\bf H2}).
 Then there exist 
a subsequence still denoted by $\{\ba_k\}$ and a vector field $\ba: B_3\times \overline\K\times \R^n \to \R^n$ such that
\[
\ba_k(x,z,\xi) \to \ba(x,z, \xi)\quad \mbox{for a.e. $x\in B_3$ and for all }\, (z,\xi)\in \overline\K\times \R^n.
\] Moreover,  $\ba$ is continuous in the $\xi$ variable. 
\end{lemma}
\begin{proof}
We first observe the following.

{\bf Claim:} For any sequence $\{\xi_n\}\subset \R^n$ with $\xi_n\to \xi$, there exists a constant $C>0$ depending only on $\xi$, $p$ and $\Lambda$ such that 
\begin{equation}\label{uniform-modulus}
\sup_{k} \sup_{(x,z)\in B_3\times \overline\K}|\ba_k(x,z,\xi_m) - \ba_k(x,z,\xi_n)|
\leq C \max{\big\{|\xi_m -\xi_n|,\, |\xi_m -\xi|^{p-1} +|\xi_n -\xi|^{p-1}\big\}} 
\end{equation}
for all $m$ and $n$ sufficiently large.
 Since the case $\xi=0$ is obvious from \eqref{structural-reference-2}, we only need to prove the claim for the case $\xi\neq 0$. Then there exists $N_0\in \N$ such that $\xi_k\in B(\xi, \frac{|\xi|}{2})$ for all $k\geq N_0$. Hence we get from the mean value property and ({\bf H1}) that 
\begin{align*}
|\ba_k(x,z,\xi_m) - \ba_k(x,z,\xi_n) |
&= |\partial_{\xi}\ba_k \big(x,z, \alpha \xi_m +(1-\alpha)\xi_n\big)|\, |\xi_m - \xi_n|\\
&\leq \Lambda |\alpha \xi_m +(1-\alpha)\xi_n|^{p-2}\, |\xi_m - \xi_n| 
\leq C |\xi_m - \xi_n|\quad \forall n,m\geq N_0,
\end{align*}
giving the claim.

Next let $(z,\xi)\in \overline\K\times \R^n$ be fixed and define
\[
\hat \ba_k(x) := \ba_k(x,z,\xi)\quad \mbox{for}\quad x\in B_3. 
\]
Then   $\{\hat \ba_k\}$ is bounded in $L^\infty(B_3)$ by \eqref{structural-reference-2}  and so there exists a subsequence depending on $(z,\xi)$ and $\hat \ba\in L^\infty(B_3)$  
such that
\[
\hat\ba_k \rightharpoonup \hat \ba \quad \mbox{weakly-* in }\, L^\infty(B_3;\R^n).
\]
Hence it follows from  condition ({\bf H2}) and Lemma~\ref{lm:sufficient-condition}  that $\hat\ba_k \longrightarrow \hat \ba$ strongly in $L^1(B_3;\R^n)$. Thus we can extract a further  subsequence, still denoted by $\{\hat \ba_k\}$, such that 
$\hat\ba_k(x) \to \hat \ba(x)$ for a.e. $x\in B_3$. 

Therefore, we infer from the diagonal process  that there exist a subsequence $\{\ba_k\}$ and a vector field $\ba: B_3\times (\overline\K\cap \Q)\times (\R^n \cap \Q^n)\to \R^n$ satisfying 
\[
\ba_k(x,z,\xi) \to \ba(x,z,\xi)
\]
for a.e. $x\in B_3$ and for all $(z,\xi)\in (\overline\K\cap \Q)\times (\R^n \cap \Q^n)$. We are going to show that $\ba$ admits an extension on $B_3\times \overline\K\times \R^n$ with the property 
\begin{equation}\label{desired-a.e.-conv}
\ba_k(x,z,\xi) \to \ba(x,z,\xi)
\quad \mbox{for a.e. $x\in B_3$ and for all $(z,\xi)\in \overline\K\times \R^n$}. 
\end{equation}
Let $(x,z,\xi)\in B_3\times \overline\K\times \R^n$ and take a sequence $\{(z_n,\xi_n)\}\subset (\overline\K\cap \Q)\times (\R^n \cap \Q^n)$ such that $(z_n,\xi_n) \to (z,\xi)$. By using \eqref{structural-reference-3} and \eqref{uniform-modulus} we obtain for all $m, n$  large that
\begin{align*}
|\ba_k(x, z_m, \xi_m) - \ba_k(x, z_n, \xi_n)|
&\leq  |\ba_k(x, z_m, \xi_m) - \ba_k(x, z_n, \xi_m)| +|\ba_k(x, z_n, \xi_m) - \ba_k(x, z_n, \xi_n)|\\
&\leq \Lambda |z_m - z_n| \, |\xi_m|^{p-1} +C \max{\big\{|\xi_m -\xi_n|,\, |\xi_m -\xi|^{p-1} +|\xi_n -\xi|^{p-1}\big\}} \quad \forall k.
\end{align*}
It follows by taking $k\to \infty$ that 
\begin{equation}\label{Cauchy-sequence}
|\ba(x, z_m, \xi_m)- \ba(x, z_n, \xi_n)|\leq \Lambda |z_m - z_n|\, |\xi_m|^{p-1}  +C \max{\big\{|\xi_m -\xi_n|,\, |\xi_m -\xi|^{p-1} +|\xi_n -\xi|^{p-1}\big\}}
\end{equation} 
for all $m, n$ sufficiently large.
 Thus, $\{\ba(x, z_n, \xi_n)\}$ is a Cauchy sequence in $\R^n$ and we define 
\[
\ba(x,z,\xi) := \lim_{n\to \infty} \ba(x, z_n, \xi_n).
\]
We note that this definition of $\ba(x,z,\xi)$ is independent of the choice of the sequence $\{(z_n,\xi_n)\}$. Indeed, if $\{(z_n',\xi_n')\}$ is another sequence in $(\overline\K\cap \Q)\times (\R^n \cap \Q^n)$ satisfying  $(z_n',\xi_n') \to (z,\xi)$, then by the same arguments leading to \eqref{Cauchy-sequence} we have
\begin{equation*}
|\ba(x, z_n, \xi_n)- \ba(x, z_n', \xi_n')|\leq \Lambda |z_n - z_n'|\, |\xi_n|^{p-1} +C \max{\big\{|\xi_n -\xi_n'|,\, |\xi_n -\xi|^{p-1} +|\xi_n' -\xi|^{p-1}\big\}}.
\end{equation*} 
Therefore, the convergent sequences $\{\ba(x, z_n, \xi_n)\}$ and $\{\ba(x, z_n', \xi_n')\}$ have the same limit.

Let us now verify \eqref{desired-a.e.-conv}. Let $(x,z,\xi)\in B_3\times \overline\K\times \R^n$ be arbitrary. Take $\{(z_n,\xi_n)\}\subset (\overline\K\cap \Q)\times (\R^n \cap \Q^n)$ be such that $(z_n,\xi_n) \to (z,\xi)$. Then the triangle inequality gives 
\begin{align*}
|\ba_k(x, z, \xi) - \ba(x, z, \xi)|
&\leq |\ba_k(x, z, \xi) - \ba_k(x, z, \xi_n)| +  |\ba_k(x, z, \xi_n) - \ba_k(x, z_n, \xi_n)|\\
& +|\ba_k(x, z_n, \xi_n) - \ba(x, z_n, \xi_n)| +|\ba(x, z_n, \xi_n) - \ba(x, z, \xi)|\quad \forall n.
\end{align*}
Moreover, it follows from \eqref{uniform-modulus} by letting $m\to \infty$ that
\begin{equation}\label{k-conti}
|\ba_k(x,z,\xi) - \ba_k(x,z,\xi_n)|
\leq C \max{\big\{|\xi -\xi_n|,\,|\xi_n -\xi|^{p-1}\big\}}.
\end{equation}
Thus, we deduce that 
\begin{align*}
|\ba_k(x, z, \xi) - \ba(x, z, \xi)|
&\leq C \max{\big\{|\xi -\xi_n|,\,|\xi_n -\xi|^{p-1}\big\}}  +  \Lambda |\xi_n|^{p-1} |z-z_n|\\
& +|\ba_k(x, z_n, \xi_n) - \ba(x, z_n, \xi_n)| +|\ba(x, z_n, \xi_n) - \ba(x, z, \xi)|\quad \forall n\geq N_0,
\end{align*}
where $N_0$ depends on $\xi$ but  independent of $k$. Consequently,
\begin{align*}
\limsup_{k\to\infty}|\ba_k(x, z, \xi) - \ba(x, z, \xi)|
&\leq C \max{\big\{|\xi -\xi_n|,\,|\xi_n -\xi|^{p-1}\big\}}  +  \Lambda |\xi_n|^{p-1} |z-z_n| \\
&\quad +|\ba(x, z_n, \xi_n) - \ba(x, z, \xi)|
\end{align*}
for all $n\geq N_0$. Letting $n\to \infty$, we conclude that $\ba_k(x, z, \xi) \to \ba(x, z, \xi)$ and hence \eqref{desired-a.e.-conv} holds true.

It remains to show that $\ba$ is continuous in the $\xi$ variable. To see this, let $\xi_n \to \xi$ in $\R^n$. Then \eqref{k-conti} is satisfied for all $k$ and so by letting $k$ tend to infinity and using \eqref{desired-a.e.-conv}  we obtain
\begin{equation*}
|\ba(x,z,\xi) - \ba(x,z,\xi_n)|
\leq C \max{\big\{|\xi -\xi_n|,\,|\xi_n -\xi|^{p-1}\big\}}\quad \,\mbox{ for all large } n.
\end{equation*}
Therefore for a.e. $x\in B_3$ and all $z\in \overline\K$, the vector field  $\xi\mapsto  \ba(x,z,\xi)$ is continuous on $\R^n$.
\end{proof}

\subsection{An approximation lemma}\label{sub:approximation-solu}

 We begin this subsection with a  result   needed for the proof of the approximation lemma (Lemma~\ref{lm:compare-solution}).

\begin{lemma}\label{lm:Minty–Browder}
Let $\omega: [0,\infty)\to \R$ be a function satisfying $\lim_{s\to 0^+}{\omega(s)}=\omega(0)=0$. For each $k$, let $\A_k: B_3\times \overline\K\times \R^n \to \R^n$ be such that
for a.e. $x\in B_3$ there hold
\begin{align}
&\big\langle\A_k(x,z,\xi) - \A_k(x,z,\eta), \xi - \eta
\big\rangle\geq 0 \qquad\qquad\qquad\quad \forall z\in \overline\K \, \mbox{ and }\,  \forall \xi, \eta\in\R^n,\label{ellipticity-ass}\\ 
& |\A_k(x,z,\xi)|   \leq \Lambda(1 +|\xi|^2)^{\frac{p-1}{2}} \,\,\qquad\qquad\qquad\qquad\qquad\forall (z,\xi)\in \overline\K\times\R^n,\label{growth-ass}\\
& |\A_k(x,z_1,\xi)-\A_k(x,z_2,\xi)|  \leq \omega(|z_1 - z_2|)(1 +|\xi|^2)^{\frac{p-1}{2}}\quad \forall z_1, z_2\in \overline\K \, \mbox{ and }\, \forall \xi\in\R^n. \label{z-continuity-ass}
\end{align}
Suppose in addition that $\A_k(x,z,\xi) \to \A(x,z, \xi)$ for a.e. $x\in B_3$ and for all $(z,\xi)\in \overline\K\times \R^n$, where $\A: B_3\times \overline\K\times \R^n \to \R^n$ is continuous in the $\xi$ variable. 
Let  $u^k\in W^{1,p}(B_3)$ be a weak solution to 
\begin{equation}\label{k-equation}
\div \A_k(x, m^k,  \nabla  u^k)=\div  \F_k\quad \mbox{in}\quad B_3
\end{equation}
with  $m^k\in L^1(B_3)$  satisfying $m^k(x)\in \overline{\K}\,$ for a.e. $x\in B_3$.
Assume that $u^k\to u$ strongly in $L^p(B_3)$, $\nabla  u^k \rightharpoonup \nabla u$ weakly in $L^p(B_3)$,  $m^k\to m$ a.e. in $B_3$, $\F_k \to 0$ strongly in $L^{\frac{p}{p-1}}(B_3; \R^n)$, 
and
\begin{equation}\label{weakly-conv-ass}
\A_k(x,m^k,  \nabla  u^k)
\rightharpoonup \bzeta\quad \mbox{weakly in } L^{\frac{p}{p-1}}(B_3; \R^n) \quad \mbox{for some}\quad \bzeta \in L^{\frac{p}{p-1}}(B_3; \R^n).
\end{equation}
 Then we have
\[
\bzeta(x) =\A\big(x, m(x), \nabla  u(x)\big)
\quad \mbox{for a.e. } x\in B_3.
\]
\end{lemma}
\begin{proof}
We shall use Minty--Browder's technique which employs monotonicity to justify passing to weak limits within a nonlinearity (see \cite{E2, LL, L}). This technique was also used  in \cite{BWZ}. 
Let $\phi \in C^\infty_0(B_3)$ be a nonnegative function. Then for any function $v\in W^{1,p}(B_3)$, we have from \eqref{ellipticity-ass} that
\begin{align*}
\int_{B_3} \left\langle\A_k(x, m^k,  \nabla  u^k)
 -\A_k(x, m^k,  \nabla  v), \nabla  u^k - \nabla v
\right\rangle  \phi\, dx\geq 0
\end{align*}
which can be rewritten as 
\begin{align}\label{MB}
&\int_{B_3} \langle  \A_k(x, m^k,  \nabla  u^k)
, \nabla  u^k 
\rangle  \phi\, dx
-\int_{B_3} \langle\A_k(x, m^k,  \nabla  u^k),  \nabla v
\rangle  \phi\, dx\\
&-\int_{B_3} \langle \A_k(x, m^k,  \nabla  v), \nabla  u^k - \nabla  v 
\rangle  \phi\, dx
\geq 0.\nonumber
\end{align}
By using $u^k\phi$ as a test function for  \eqref{k-equation}, we see that the first term in \eqref{MB}  is the same as 
\[
-\int_{B_3} \langle\A_k(x, m^k,  \nabla  u^k)
, \nabla  \phi
\rangle  u^k dx
+ \int_{B_3} \langle  \F_k,\phi \nabla  u^k  +u^k \nabla  \phi \rangle \, dx.
\]
Therefore,  inequality \eqref{MB} becomes
\begin{align}\label{Minty-ineq}
&-\int_{B_3} \langle\A_k(x, m^k,  \nabla  u^k)
, \nabla  \phi
\rangle  u^k dx + \int_{B_3} \langle  \F_k,\phi \nabla  u^k  +u^k \nabla  \phi \rangle \, dx
\nonumber\\
&-\int_{B_3} \langle\A_k(x, m^k,  \nabla  u^k),  \nabla v
\rangle  \phi dx -\int_{B_3} \langle \A_k(x, m^k,  \nabla  v), \nabla  u^k -\nabla v  
\rangle  \phi\, dx
 \geq 0.
\end{align}
Notice that since
\begin{align*}
 |\A_k(x, m^k,  \nabla  v) -
\A(x, m,  \nabla  v)| 
&\leq |\A_k(x, m^k,  \nabla  v) -
\A_k(x, m,  \nabla  v)| + |\A_k(x, m,  \nabla  v) -
\A(x, m,  \nabla  v)|\\
&\leq \omega(|m^k -m|) (1 +|\nabla v|^2)^{\frac{p-1}{2}} +|\A_k(x, m,  \nabla  v) -
\A(x, m,  \nabla  v)|,
\end{align*}
we get $\A_k(x, m^k,  \nabla  v) \to 
\A(x, m,  \nabla  v)$ for a.e. $x\in B_3$. Hence we conclude from condition \eqref{growth-ass}  and the Lebesgue's dominated convergence theorem   that 
$\A_k(x, m^k,  \nabla  v) \to 
\A(x, m,  \nabla  v)$ strongly in $L^{\frac{p}{p-1}}(B_3; \R^n)$. Therefore,
\begin{align*}
\lim_{k\to \infty}\int_{B_3} \langle \A_k(x, m^k,  \nabla  v), \nabla  u^k 
-\nabla v\rangle  \phi\, dx
=\int_{B_3} \langle \A(x, m,  \nabla  v), \nabla  u -\nabla v 
\rangle  \phi\, dx.
\end{align*}
Using this and  assumption \eqref{weakly-conv-ass}, we can pass to the limits in \eqref{Minty-ineq} to obtain
\begin{align}\label{limiting-Minty}
-\int_{B_3} \langle \bzeta, \nabla  \phi 
\rangle u\, dx
-\int_{B_3} \langle \bzeta, \nabla  v 
\rangle \phi\, dx
-\int_{B_3} \langle \A(x, m,  \nabla  v), \nabla  u -\nabla v
\rangle  \phi\, dx
\geq 0.
\end{align}
On the other hand, by choosing $u\phi$ as a test function for equation  \eqref{k-equation} and passing to the limits, we get 
$\int_{B_3}{\langle \bzeta, \nabla (u\phi)\rangle\, dx}=0$ which yields
$-\int_{B_3}{\langle \bzeta, \nabla \phi\rangle u \, dx}=\int_{B_3}{\langle \bzeta, \nabla u\rangle\phi\, dx}$.
Hence we can rewrite \eqref{limiting-Minty} as 
\begin{align*}
\int_{B_3}\langle \bzeta- \A(x, m,  \nabla  v),  \nabla u-\nabla v
\rangle  \phi\, dx\geq 0.
\end{align*}
By taking $v = u \pm\alpha w$  and letting $\alpha\to 0^+$, one easily  deduces from the above inequality and the assumption $\A$ being continuous in the $\xi$ variable that
\begin{align*}
\int_{B_3}\langle \bzeta- \A(x, m,  \nabla  u),  \nabla w
\rangle  \phi\, dx=0
\end{align*}
for all functions $w\in W^{1,p}(B_3)$ and all nonnegative functions $\phi \in C^\infty_0(B_3)$. It then follows that $\bzeta= \A(x, m,  \nabla  u)\,$ a.e. in $B_3$.
\end{proof}

The following approximation lemma plays a central role in our proof of Theorem~\ref{main-result}. It is crucial that the  constant $\delta>0$ is independent of the parameters $\lambda$ and $\theta$. We shall prove it  by extending   the compactness argument used in
 \cite[Lemma~2.11]{HNP1} and for this  purpose we define
\begin{equation*}\label{oscillating-fn}
d_{\A,\ba}(x) := \sup_{z\in \overline\K}\sup_{ \xi\neq 0}\frac{|\A(x,z,\xi) - \ba(x,z,\xi)|}{|\xi|^{p-1}}.
\end{equation*}

\begin{lemma}\label{lm:compare-solution} Let $\A$ satisfy \eqref{structural-reference-1}--\eqref{structural-reference-3}, $\ba \in \G_{B_3}(\eta)$, and $M_0>0$.
For any $\e>0$, there exists $\delta>0$ depending only on $\e$, $\Lambda$, $p$,   $\eta$,  $n$, $\K$ and $M_0$  such that:  if $\lambda>0$, $ 0<\theta \leq 1$,
\begin{equation*}
\fint_{B_3}  d_{\A,\ba}(x) \, dx  \leq \delta, \qquad \fint_{B_4} |\F|^{\frac{p}{p-1}} \, dx  \leq \delta,
\end{equation*}
and $u\in W_{loc}^{1,p}(B_4)$ is a weak solution of 
\begin{equation}\label{eq-for-u}
\div \Big[\frac{\A(x,\lambda\theta u,\lambda \nabla u)}{\lambda^{p-1}}\Big] =  \div \F\quad \text{in}\quad B_4
\end{equation}
 satisfying
\begin{equation}\label{gradone}
 \Big(\fint_{B_4}{|u|^p\, dx}\Big)^{\frac1p}\leq \frac{M_0}{\lambda \theta} \quad\mbox{ and }\quad 
\fint_{B_4}{ |\nabla  u|^p\, dx} \leq 1,
\end{equation}
and $v\in  W^{1,p}(B_3)$  is a weak solution of
\begin{equation}\label{eq-for-v}
\left \{
\begin{array}{lcll}
\div \Big[\frac{\ba(x,\lambda\theta v,\lambda \nabla v)}{\lambda^{p-1}}\Big] &=&  0 \quad &\text{in}\quad B_3, \\
v & =& u\quad &\text{on}\quad \partial B_3,
\end{array}\right.
\end{equation}
 then 
\begin{equation}\label{u-v-close}
\int_{B_3}{|u - v|^p\, dx}\leq \e^p.
\end{equation}
\end{lemma}
\begin{proof}
We  prove \eqref{u-v-close} by contradiction. Suppose that estimate \eqref{u-v-close} is not true. Then there exist $\e_0,\, p, \, \Lambda,  \,  \eta, \, n,\, \K, \, M_0$, 
 sequences of  positive numbers  $\{\lambda_k\}_{k=1}^\infty$ and  $\{\theta_k\}_{k=1}^\infty$ with $0<\theta_k\leq 1$, sequences  $\{\A_k\}_{k=1}^\infty$ and $\{\ba_k\}_{k=1}^\infty$ with 
$ \A_k$ satisfying the structural conditions \eqref{structural-reference-1}--\eqref{structural-reference-3} and 
$\ba_k\in \G_{B_3}(\eta)$, 
 and  sequences of functions $\{\F_k\}_{k=1}^\infty$,   $\{u^k\}_{k=1}^\infty$ 
such that
\begin{equation*}
\sup_{\xi\neq 0} \sup_{x_0:\, B_r(x_0)\subset B_3 } \fint_{B_r(x_0)}\frac{ \big|\ba_k(x,z,\xi) - (\ba_k)_{B_r(x_0)}(z,\xi)\big|}{|\xi|^{p-1}} \,dx\leq \eta(z,r),
\end{equation*}
\begin{equation}\label{Theta_k-f_k-condition}
\fint_{B_3} d_{\A_k,\ba_k}(x) \, dx  \leq \frac{1}{k}, \qquad
\fint_{B_4}|\F_k|^{\frac{p}{p-1}} \, dx  \leq \frac{1}{k}, 
\end{equation}
 $u^k\in W^{1,p}_{loc}(B_4)$ is a weak solution of 
\begin{equation*}
\div \Big[\frac{\A_k(x,\lambda_k\theta_k u^k,\lambda_k\nabla u^k)}{\lambda_k^{p-1}}\Big] = \div  \F_k \quad \text{in}\quad B_4
\end{equation*}
with 
\begin{equation}\label{gradient-bounded-ass}
 \Big(\fint_{B_4}{|u^k|^p\, dx}\Big)^{\frac1p}\leq \frac{M_0}{\lambda_k \theta_k} \quad\mbox{ and }\quad 
\fint_{B_4}{ |\nabla  u^k|^p\, dx} \leq 1,
\end{equation}
\begin{equation}\label{contradiction-conclusion}
\int_{B_3} |u^k - v^k|^p \, dx  > \e_0^p  \quad\mbox{for all } k.
\end{equation}
Here  $v^k\in W^{1,p}(B_3)$  is a weak solution  of
\begin{equation*}
\left \{
\begin{array}{lcll}
\div \Big[\frac{\ba_k(x,\lambda_k\theta_k v^k,\lambda_k\nabla v^k)}{\lambda_k^{p-1}}\Big] &=&  0 \quad &\text{in}\quad B_3, \\
v^k & =& u^k\quad &\text{on}\quad \partial B_3.
\end{array}\right.
\end{equation*}

Let us set
\begin{align*}
\alpha_k := \lambda_k \theta_k, \quad \hat \A_k(x,z, \xi) :=\frac{\A_k(x,z, \lambda_k\xi)}{\lambda_k^{p-1}} \quad \mbox{and}\quad \hat \ba_k(x,z, \xi) :=\frac{\ba_k(x,z, \lambda_k\xi)}{\lambda_k^{p-1}}.
\end{align*}
Then, we still have 
\[d_{\hat\A_k, \hat \ba_k}(x)=d_{\A_k,\ba_k}(x)
\]
and
\begin{equation}\label{k-osc-ref-eq}
\sup_{ \xi\neq 0} \sup_{x_0:\, B_r(x_0)\subset B_3 } \fint_{B_r(x_0)}\frac{ \big|\hat\ba_k(x,z,\xi) - (\hat\ba_k)_{B_r(x_0)}(z,\xi)\big|}{|\xi|^{p-1}} \,dx\leq \eta(z,r).
\end{equation}
Moreover, $u^k$ is a weak solution of 
\begin{equation}\label{eq-u_k}
\div \hat \A_k(x,\alpha_k u^k,\nabla u^k) = \div  \F_k\quad \text{in}\quad B_4,
\end{equation}
and $v^k$  is a weak solution  of
\begin{equation*}
\left \{
\begin{array}{lcll}
\div \hat \ba_k(x,\alpha_k v^k,\nabla v^k) &=&  0 \quad &\text{in}\quad B_3, \\
v^k & =& u^k\quad &\text{on}\quad \partial B_3.
\end{array}\right.
\end{equation*}

Using Sobolev's inequality, Lemma~\ref{lm:bound-solution},  $\Big(\fint_{B_4}{|u^k|^p\, dx}\Big)^{\frac1p}\leq \frac{M_0}{\alpha_k}$ and \eqref{Theta_k-f_k-condition}, we obtain
\begin{align*}
\int_{B_3} |u^k-v^k|^p\, dx
\leq C\int_{B_3} |\nabla u^k-\nabla v^k|^p\, dx 
\leq C(p,n,\Lambda) \int_{B_4} \big(|u^k|^p +|\F_k|^{\frac{p}{p-1}} \big)\, dx
\leq  C\Big[ \big(\frac{M_0}{\alpha_k}\big)^p +k^{-1}\Big].
\end{align*}
Thus we infer from \eqref{contradiction-conclusion}  that
\[
  \alpha_k\leq \frac{M_0}{C^{-1}\e_0 -k^{\frac{-1}{p}}},
\]
and so the sequence $\{\alpha_k\}$ is  bounded. From this, Lemma~\ref{coeff-compactness}
and  by taking subsequences if necessary, we see  that there exist a constant  $\alpha\in [0,\infty)$  and a vector field $\hat\ba: B_3\times \overline\K\times \R^n \to \R^n$ being continuous in the $\xi$ variable such that
 $\alpha_k \to \alpha$ and 
$\hat\ba_k(x,z,\xi) \to \hat\ba(x,z, \xi)$ for a.e. $x\in B_3$ and for all $(z,\xi)\in \overline\K\times \R^n$. Moreover,  \eqref{Theta_k-f_k-condition} implies that, up to a subsequence,  
$d_{\A_k,\ba_k}(x) \to 0$ for a.e. $x\in B_3$. Thus, we also have $\hat \A_k(x,z,\xi) \to \hat\ba(x,z, \xi)$ for a.e. $x\in B_3$ and for all $(z,\xi)\in \overline\K\times \R^n$.
 
By using the pointwise convergence, it can be verified that
$\hat \ba$ satisfies conditions \eqref{structural-reference-1}--\eqref{structural-reference-3}.
We are going to derive a contradiction by proving the following claim.

\textbf{Claim.} There are subsequences $\{u^{k_m}\}$ and $\{v^{k_m}\}$ such that $u^{k_m} - v^{k_m} \to 0$ in $L^p(B_3)$ as $m\to\infty$.

Let us consider the case $\alpha>0$ first. Then, thanks to \eqref{gradient-bounded-ass}, the sequence $\{u^k\}$ is bounded in $W^{1,p}(B_3)$. Likewise, by using  \eqref{v-L^p-est} and  
Proposition~\ref{pro:gradient-est} with $\A(x,z,\xi)\equiv \hat\ba_k(x, \alpha_k z, \xi)$ and $\F\equiv 0$, we also have  that the sequence $\{v^k\}$ is  bounded in  $W^{1,p}(B_3)$.
Therefore there exist  subsequences, still denoted by $\{u^k\}$ and $\{v^k\}$, and  functions $u, v\in W^{1,p}(B_3)$ such that 
\begin{equation*}
 \left \{ \  
\begin{array}{lll}
&u^k \to u  \mbox{ a.e. in } B_3, \quad u^k \to u  \mbox{ strongly in } L^p(B_3),\quad \nabla  u^k \rightharpoonup \nabla  u \text{ weakly in } L^p(B_3),\\
&v^k \to v  \mbox{ a.e. in } B_3, \, \quad v^k \to v  \mbox{ strongly in } L^p(B_3),\, \quad \nabla  v^k \rightharpoonup \nabla  v \text{ weakly in } L^p(B_3).
\end{array} \right . 
\end{equation*}
In particular, we have
\begin{equation}\label{range-and-boundary}
u(x), \,v(x)\in   \frac{1}{\alpha} \overline\K \,\, \mbox{ for a.e. } x\in B_3, \quad \text{and} \quad
 u = v \quad \text{on} \quad  \partial B_3.
\end{equation}
Also as the sequence $\Big\{\hat\A_k(x,\alpha_k u^k,\nabla u^k)\Big\}$ is bounded in $L^{\frac{p}{p-1}}(B_3; \R^n)$, by taking a subsequence there exists $\bzeta \in L^{\frac{p}{p-1}}(B_3; \R^n)$ such that 
\[\hat\A_k(x,\alpha_k u^k,\nabla u^k) \rightharpoonup \bzeta \, \mbox{ weakly in }\,  L^{\frac{p}{p-1}}(B_3; \R^n).
\]
 But by  applying Lemma~\ref{lm:Minty–Browder}
for $m_k(x) \rightsquigarrow  \alpha_k u^k(x)$ and $m(x) \rightsquigarrow  \alpha u(x)$, we obtain $\bzeta \equiv \hat\ba(x,\alpha u,\nabla u)$. That is,
\[
\hat\A_k(x,\alpha_k u^k,\nabla u^k)\rightharpoonup \hat\ba(x,\alpha u,\nabla u)\,\mbox{ weakly in } L^{\frac{p}{p-1}}(B_3; \R^n).
\]
Therefore,
\begin{align}\label{eq:takecare-limit-I}
 \lim_{k\to\infty}\int_{B_3}{ \big\langle \hat\A_k(x,\alpha_k u^k,\nabla u^k), \nabla \varphi\big\rangle dx }= \int_{B_3}{ \big\langle \hat\ba(x,\alpha u,\nabla u), \nabla \varphi\big\rangle dx } \quad \mbox{for all}\quad \varphi\in C_0^\infty(B_3).
 \end{align}
Thus  by passing $k\to\infty$ for  equation \eqref{eq-u_k}, one sees that $u$ is a weak solution of the equation
\begin{equation}\label{limit-equation}
\div \hat\ba(x,\alpha u,\nabla u) =0\quad\mbox{in}\quad B_3.
\end{equation}
Similarly, $v$ is a weak solution of
\begin{equation*}
\div \hat\ba(x,\alpha v,\nabla v) =0\quad\mbox{in}\quad B_3.
\end{equation*}
Hence due to \eqref{range-and-boundary} and by the uniqueness of the weak solution to equation \eqref{limit-equation} as explained in Remark~\ref{existen
ce+uniqueness}, we conclude that 
$u \equiv  v$ in $B_3$. It follows that $u^k-v^k \to u-v=0$ strongly in $L^p(B_3)$.

Now, consider the case $\alpha =0$, that is, $\alpha_k \to 0$.  Let $\bar u^k :=\alpha_k u^k$, $\bar v^k :=\alpha_k  v^k$, $w^k := u^k - u^k_{B_3}$ 
and $h^k := v^k - u^k_{B_3}$, where $u^k_{B_3} := \fint_{B_3}{u^k(x) \, dx}$. Then $w^k\in W_{loc}^{1,p}(B_4)$ is a weak solution of 
\begin{equation}\label{eq-w_k}
\div \hat\A_k(x,\bar u^k,\nabla w^k) = \div \F_k \quad \text{in}\quad B_4
\end{equation}
and   $h^k\in W^{1,p}(B_3)$  is a weak solution  of
\begin{equation}\label{eq-h_k}
\left \{
\begin{array}{lcll}
\div \hat\ba_k(x,\bar v^k,\nabla h^k) &=&  0 \quad &\text{in}\quad B_3, \\
h^k & =& w^k\quad &\text{on}\quad \partial B_3.
\end{array}\right.
\end{equation}

By applying Proposition~\ref{pro:gradient-est}  for $w \rightsquigarrow  v^k$, $\varphi \rightsquigarrow  u^k$, $\F \equiv 0$ and using \eqref{gradient-bounded-ass}, we get 
\begin{equation}\label{eq-grad-v-k}
\int_{B_3} |\nabla  v^k |^p \, dx\leq C \int_{B_3} |\nabla  u^k|^p\, dx 
\leq C\quad \mbox{for all  } k.
\end{equation}
Consequently,
\begin{equation*}
 \int_{B_3}|\nabla  u^k - \nabla  v^k|^p \, dx 
\leq  2^{p-1}\left[\int_{B_3} |\nabla  u^k|^p\, dx 
+ \int_{B_3} |\nabla  v^k |^p \, dx\right]\leq C
\end{equation*}
which together with the Sobolev's inequality gives 
\begin{equation}\label{difference-control}
\int_{B_3} |u^k - v^k|^p \, dx 
\leq C\int_{B_3}|\nabla  (u^k - v^k)|^p \, dx \leq C.
\end{equation}

 Notice that on one hand the Poincar\'e inequality gives
 \[
 \int_{B_3} |w^k|^p \, dx=\int_{B_3} |u^k - u^k_{B_3}|^p \, dx 
\leq C\int_{B_3}|\nabla  u^k|^p \, dx \leq C.
 \]
On the other hand, by employing the Poincar\'e inequality, \eqref{eq-grad-v-k} and \eqref{difference-control} we obtain
\begin{align*}
\|h^k\|_{L^p(B_3)}
&= \|v^k - u^k_{B_3}\|_{L^p(B_3)}\leq \|v^k - v^k_{B_3}\|_{L^p(B_3)} + |B_3|^{\frac1p} \, |u^k_{B_3}-v^k_{B_3}|\\
&\leq C \|\nabla  v^k\|_{L^p(B_3)} +  \|u^k - v^k\|_{L^p(B_3)}\leq C.
\end{align*}
Therefore, $\{w^k\}$ and $\{h^k\}$ are bounded sequences in $W^{1,p}(B_3)$. 
Moreover, $\{\bar{v}^k\}$ is bounded in $W^{1,p}(B_3)$ owing  to \eqref{eq-grad-v-k} and
\[
\|v^k\|_{L^p(B_3)} \leq \|u^k\|_{L^p(B_3)} 
+\|u^k-v^k\|_{L^p(B_3)}
\leq |B_4|^{\frac1p} \, \frac{M_0}{\alpha_k} +C.
\]
Consequently there are subsequences, still denoted by $\{w^k\}$, $\{h^k\}$ and $\{\bar v^k\}$ and three functions $w, h, \bar v\in W^{1,p}(B_3)$   such that
\begin{equation*}
 \left \{ \  
\begin{array}{lll}
&w^k \to w \mbox{ a.e. in } B_3, \quad w^k \to w \mbox{ strongly in } L^p(B_3),\quad \nabla  w^k \rightharpoonup \nabla  w \text{ weakly in } L^p(B_3),\\
&h^k \to h  \mbox{ a.e. in } B_3, \, \, \quad h^k \to h  \mbox{ strongly in } L^p(B_3),\, \, \quad \nabla  h^k \rightharpoonup \nabla  h \text{ weakly in } L^p(B_3),\\
&\bar v^k \to \bar v  \mbox{ a.e. in } B_3,\, \, \,  \quad \bar v^k \to \bar v \mbox{ strongly in } L^p(B_3),\, \, \, \quad \nabla  \bar v^k \rightharpoonup \nabla  \bar v \text{ weakly in } L^p(B_3).
\end{array} \right . 
\end{equation*}
Since $\nabla \bar{v}^k =\alpha_k \nabla  v^k \to 0$  
in $L^p(B_3)$ thanks to \eqref{eq-grad-v-k}, we infer further that   $\nabla  \bar{v}^k \to \nabla \bar v\equiv 0$ strongly in $L^p(B_3)$. Thus, $\bar v$ is a constant function. As $\nabla \bar{u}^k =\alpha_k \nabla  u^k \to 0$  
in $L^p(B_3)$ and 
\begin{align*}
\| \bar u^k -\bar v\|_{L^p(B_3)}
&\leq \| \bar u^k -\bar v^k\|_{L^p(B_3)} + \| \bar v^k -\bar v\|_{L^p(B_3)}\\
&=\alpha_k \|  u^k - v^k\|_{L^p(B_3)} + \| \bar v^k -\bar v\|_{L^p(B_3)}
\leq C \alpha_k  + \| \bar v^k -\bar v\|_{L^p(B_3)},
\end{align*}
we also have $\bar u^k \to \bar v$ strongly in $W^{1,p}(B_3)$. By taking a further subsequence, we can assume that
$\bar u^k(x) \to \bar v$  a.e. in $B_3$.
 
It follows from Lemma~\ref{lm:Minty–Browder}
for $m_k(x) \rightsquigarrow  \bar u^k(x)$ and $m(x) \rightsquigarrow  \bar v$ that 
\[
\hat\A_k(x,\bar u^k,\nabla w^k) \rightharpoonup \hat\ba(x,\bar v,\nabla w)\,\mbox{ weakly in } L^{\frac{p}{p-1}}(B_3; \R^n)
\]
up to a subsequence. Then as  in \eqref{eq:takecare-limit-I}, one gets for all $\varphi\in C_0^\infty(B_3)$ that
\begin{align*}
 \lim_{k\to\infty}\int_{B_3}{\big\langle\hat\A_k(x,\bar u^k,\nabla w^k), \nabla \varphi\big\rangle dx }
 &= \int_{B_3}{ \big\langle \hat\ba(x,\bar v,\nabla w), \nabla \varphi\big\rangle dx }.
 \end{align*}
Hence by passing to the limit in equation \eqref{eq-w_k}, we  
conclude that $w$ is a weak solution of
\begin{equation*}
\div\hat\ba(x,\bar v,\nabla w)=0 \,\,\text{ in}\quad B_3.
\end{equation*}
Likewise, we deduce from \eqref{eq-h_k} that $h$ is  a weak solution of
\begin{equation}\label{limit-equation-case2}
\left \{
\begin{array}{lcll}
\div \hat\ba(x,\bar v,\nabla h) &=&0 &\text{ in}\quad B_3, \\
h & =& w &\text{ on}\quad \partial B_3.
\end{array}\right.
\end{equation}
By the uniqueness of the weak solution to  equation \eqref{limit-equation-case2}, we conclude that $h \equiv w$ in $B_3$. This gives, again, $u^k-v^k=w^k - h^k\to 0$ in $L^p(B_3)$ as $k\to\infty$.
Therefore, we have proved the Claim which contradicts \eqref{contradiction-conclusion}. Thus the proof of \eqref{u-v-close} is complete.
\end{proof}

\section{Approximating gradients of solutions }\label{approximation-gradient}
 Throughout this section,
 let $\omega: [0,\infty)\to [0,\infty)$ be the function defined  by
\begin{equation}\label{modified-Lipschitz}
\omega(r)= \left \{
\begin{array}{lcll}
r \Lambda  \qquad  \text{if}\quad 0\leq r \leq 2, \\
2\Lambda  \qquad\text{if}\quad r>2.
\end{array}\right.
\end{equation}
Notice that if $\ba$ satisfies \eqref{structural-reference-2}--\eqref{structural-reference-3}, then we obtain from the definition of $\omega$ that
 \begin{equation}\label{modified-modulus}
  |\ba(x,z_1, \xi) -\ba(x, z_2,\xi)| \leq \omega(|z_1 - z_2|)\, |\xi|^{p-1}  \quad \forall z_1, z_2\in  \overline{\K}.
 \end{equation}

Our aim is to approximate  $\nabla u$ 
by a good gradient in $L^p$ norm, and the following lemma is the starting point for that purpose.
\begin{lemma}\label{lm:compare-gradient-solution}
Assume that $\A$ satisfies \eqref{structural-reference-1}--\eqref{structural-reference-2}, $\ba \in \G_{B_3}(\eta)$,  $M_0>0$, $\lambda>0$, and $0<\theta\leq 1$. Let   $u\in  W^{1,p}_{loc}(B_4)$
be a weak solution  of \eqref{eq-for-u}
with  $\fint_{B_4}{|\nabla u|^p\, dx}\leq 1$ and $\|\F\|_{L^{\frac{p}{p-1}}(B_4)}\leq 1$.
Then for any weak solution  $v\in  W^{1,p}(B_3)$ of  
\eqref{eq-for-v}
satisfying  $\|v\|_{L^\infty(B_3)}\leq \frac{M_0}{\lambda \theta}$, we have: 
\begin{enumerate}
 \item[(i)] If $p\geq 2$, then 
 \begin{align*}
\int_{B_2}{ |\nabla u - \nabla v|^p dx }
\leq C \left\{\int_{B_{\frac52}}{\Big[\omega\big(|\lambda\theta( u-v)|\big) + d_{\A,\ba}(x)\Big]^{\frac{p}{p-1}}   dx}  +  \int_{B_{\frac52}}{|\F|^{\frac{p}{p-1}} dx}\right\}
 +
C \|u-v\|_{L^p(B_{\frac52})}.
 \end{align*}
\item[(ii)] If $1<p<2$, then 
\begin{align*}
\int_{B_2} |\nabla u -\nabla  v|^p  dx 
&\leq  \frac{C}{ \sigma^{\frac{p}{p-1}}} \left\{\int_{B_{\frac52}}{\Big[\omega\big(|\lambda\theta( u-v)|\big) + d_{\A,\ba}(x)\Big]^{\frac{p}{p-1}}    dx}+
\int_{B_{\frac52}}{|\F|^{\frac{p}{p-1}}  dx}\right\}
\\
&\quad +C\sigma^{\frac{p}{2-p}}  + \frac{C}{\sigma}\|u-v\|_{L^p(B_{\frac52})}\qquad \qquad\qquad\qquad\qquad\mbox{for every $\sigma>0$ small.}
\end{align*}
\end{enumerate}
Here  the constant $C>0$ depends only on $p$, $n$, $\eta$, $\Lambda$, $\K$ and $M_0$.
\end{lemma}
\begin{proof}
Observe that if we let $\bar v(y) := \lambda \theta v(y/\theta)$, then  $\|\bar v\|_{L^\infty(B_{3\theta})}\leq M_0$ and
  $\bar v $ is a weak solution of $\div \ba(y,  \bar v , \nabla  \bar v)=0$ in $B_{3\theta}$. Thus   assumption ({\bf H3})  about  interior $W^{1,\infty}$-estimates gives 
 \begin{equation}\label{Lipschitz-1}
 \| \nabla  \bar v\|_{L^\infty(B_{\frac{5\theta}{2}})}^p\leq C(p,n,\eta, \Lambda, \K, M_0) \fint_{B_{3\theta}} |\nabla  \bar v|^p \, dx.
 \end{equation}
 On the other hand, it follows from Proposition~\ref{pro:gradient-est} and  the assumptions  that 
 \[
 \|\nabla v\|_{L^p(B_3)}\leq C \|\nabla u\|_{L^p(B_3)} \leq C(p,n, \Lambda).
 \]
Therefore, we  have from \eqref{Lipschitz-1} by rescaling back that
 \begin{equation}\label{Lipschitz-2}
 \| \nabla  v\|_{L^\infty(B_{\frac52})}^p\leq C(p,n,\eta, \Lambda, \K, M_0) \fint_{B_{3}} |\nabla   v|^p \, dx \leq C(p,n,\eta, \Lambda, \K, M_0).
 \end{equation}
 Next for convenience,  set
\[
\hat \A (x, z,\xi) := \frac{\A(x,z,\lambda \xi)}{\lambda^{p-1}}\quad \mbox{and}\quad \hat \ba (x, z,\xi) := \frac{\ba(x,z,\lambda \xi)}{\lambda^{p-1}}.
\]
Let $\varphi$ be the standard nonnegative cut-off function which is $1$ on $B_2$ and $\text{supp}(\varphi)\subset B_{\frac52}$.
Then by using $\varphi^p (u-v)$ as a test function in the equations for $u$ and $v$, we have
\begin{align*}
&\int_{B_{\frac52}}{ \langle \hat \A(x, \lambda \theta u, \nabla u) , \nabla u - \nabla v\rangle \varphi^p \, dx }
=-p\int_{B_{\frac52}}{ \langle \hat \A(x, \lambda \theta u, \nabla u) , \nabla \varphi\rangle(u-v)\varphi^{p-1} \, dx }\\
&+p\int_{B_{\frac52}}{ \langle \hat \ba(x, \lambda\theta v, \nabla v) , \nabla \varphi\rangle(u-v)\varphi^{p-1} \, dx }
+\int_{B_{\frac52}}{ \langle \hat \ba(x,\lambda\theta v,\nabla v), \nabla u - \nabla v\rangle \varphi^p  dx }\\
&+\int_{B_{\frac52}}{\langle \F, \nabla u - \nabla v\rangle  \varphi^p dx}
+p\int_{B_{\frac52}}{\langle \F, \nabla \varphi\rangle (u-v) \varphi^{p-1} dx}.
\end{align*}
This gives
\begin{align*}
 I &:=\int_{B_{\frac52}}{ \langle \hat\A(x, \lambda\theta u, \nabla u) -\hat \A(x,\lambda\theta u,\nabla v), \nabla u - \nabla v \rangle \varphi^p\, dx }\\
 &=p\int_{B_{\frac52}}{ \langle \hat \ba(x, \lambda\theta v, \nabla v)- \hat \A(x, \lambda\theta u, \nabla u) , \nabla \varphi\rangle(u-v)\varphi^{p-1} \, dx }\\
 &\quad  +\int_{B_{\frac52}}{ \langle \hat \ba(x, \lambda\theta v, \nabla v) -\hat \ba(x,\lambda\theta u,\nabla v), \nabla u - \nabla v \rangle \varphi^p \, dx}\\
 &\quad -\int_{B_{\frac52}}{ \langle\hat \A(x, \lambda\theta u, \nabla v) -\hat \ba(x,\lambda\theta u,\nabla v), \nabla u - \nabla v\rangle \varphi^p \, dx }\\
&\quad +\int_{B_{\frac52}}{\langle \F, \nabla u - \nabla v\rangle \varphi^p\, dx}
+p\int_{B_{\frac52}}{\langle \F, \nabla \varphi\rangle (u-v) \varphi^{p-1} dx}.
\end{align*}
We deduce   from this, the structural conditions, \eqref{modified-modulus} and  \eqref{Lipschitz-2}  that
\begin{align}\label{homo-est1}
I&\leq p\Lambda \int_{B_{\frac52}}{ \big( |\nabla v|^{p-1} + |\nabla u|^{p-1}\big) |\nabla \varphi| |u-v|\varphi^{p-1} \, dx }\\
&\quad +\Lambda  \int_{B_{\frac52}}{\frac{\omega\big(|\lambda\theta( u-v)|\big)}{\lambda^{p-1}}  |\nabla (\lambda v)|^{p-1} | \nabla u - \nabla v|  \varphi^p\, dx}
+ \int_{B_{\frac52}}{d_{\A,\ba}(x)  |\nabla v|^{p-1} | \nabla u - \nabla v|  \varphi^p\, dx}\nonumber\\
&\quad+\int_{B_{\frac52}}{|\F| |\nabla u - \nabla v| \varphi^p\, dx}
+p\int_{B_{\frac52}}{|\F|| | \nabla \varphi| |u-v| \varphi^{p-1} dx}\nonumber\\
&\leq C\int_{B_{\frac52}}{\Big[\omega\big(|\lambda\theta( u-v)|\big)|  + d_{\A,\ba}(x)\Big] | \nabla u - \nabla v|  \varphi^p \, dx}+ \int_{B_{\frac52}}{|\F| |\nabla u - \nabla v|  \varphi^p \, dx}\nonumber\\
& \quad + C\Big(\|\nabla v\|_{L^p(B_{\frac52})} + \|\nabla u \|_{L^p(B_{\frac52})} +\|\F\|_{L^{\frac{p}{p-1}}(B_{\frac52})} \Big) \|u-v\|_{L^p(B_{\frac52})} \nonumber\\
&\leq \frac{2\sigma}{p} \int_{B_{\frac52}}{| \nabla u - \nabla v|^p\varphi^p\, dx} +C\frac{p-1}{p \sigma^{\frac{1}{p-1}}} \int_{B_{\frac52}}{\Big[\omega\big(|\lambda\theta( u-v)|\big)
 + d_{\A,\ba}(x)\Big]^{\frac{p}{p-1}}  dx}\nonumber\\
&\quad 
+\frac{p-1}{p \sigma^{\frac{1}{p-1}}}\int_{B_{\frac52}}{|\F|^{\frac{p}{p-1}} dx}
+ C \|u-v\|_{L^p(B_{\frac52})} \nonumber
\end{align}
for any $\sigma>0$.

Now if $p\geq 2$, then  \eqref{structural-reference-1}  implies
$\Lambda^{-1}\int_{B_{\frac52}}{ |\nabla u - \nabla v|^p\varphi^p dx }\leq I$. Hence by combining with \eqref{homo-est1} and choosing $\sigma>0$ sufficiently small, we obtain
\begin{align*}
\int_{B_{\frac52}}{ |\nabla u - \nabla v|^p\varphi^p dx }
\leq C \left \{\int_{B_{\frac52}}{\Big[\omega\big(|\lambda\theta( u-v)|\big)  + d_{\A,\ba}(x)\Big]^{\frac{p}{p-1}}   dx} +  \int_{B_{\frac52}}{|\F|^{\frac{p}{p-1}} dx}\right\}
 +
C \|u-v\|_{L^p(B_{\frac52})}
\end{align*}
giving  $(i)$.
On the other hand,  if $1<p<2$ then  Lemma~\ref{simple-est} yields
\[
c \tau^{\frac{2-p}{p}} \int_{B_{\frac52}} |\nabla u -\nabla  v|^p \varphi^p dx 
- \tau^{\frac{2}{p}} \int_{B_{\frac52}} |\nabla u|^p \varphi^p dx
 \leq I
\]
for all $\tau>0$ small. By combining this with the assumptions and \eqref{homo-est1} we deduce that
\begin{align*}
(\tau^{\frac{2-p}{p}}-\sigma)  \int_{B_{\frac52}} |\nabla u-\nabla  v|^p\varphi^p  dx 
&\leq \frac{C}{ \sigma^{\frac{1}{p-1}}} \left\{\int_{B_{\frac52}}{\Big[\omega\big(|\lambda\theta( u-v)|\big)  + d_{\A,\ba}(x)\Big]^{\frac{p}{p-1}}   dx}+
\int_{B_{\frac52}}{|\F|^{\frac{p}{p-1}}  dx}\right\}\\
&\quad +C\tau^{\frac{2}{p}}   + C\|u-v\|_{L^p(B_{\frac52})}.
\end{align*}
It then follows by taking $\tau^{\frac{2-p}{p}}=2\sigma$ that
\begin{align*}
\int_{B_{\frac52}} |\nabla u -\nabla  v|^p\varphi^p  dx 
&\leq  \frac{C}{ \sigma^{\frac{p}{p-1}}} \left\{\int_{B_{\frac52}}{\Big[\omega\big(|\lambda\theta( u-v)|\big)  + d_{\A,\ba}(x)\Big]^{\frac{p}{p-1}}    dx}+
\int_{B_{\frac52}}{|\F|^{\frac{p}{p-1}}  dx}\right\}\\
&\quad +C\sigma^{\frac{p}{2-p}}  + \frac{C}{\sigma}\|u-v\|_{L^p(B_{\frac52})}
\end{align*}
for every $\sigma>0$ small.
\end{proof}

As a consequence of Lemma~\ref{lm:compare-solution} and Lemma~\ref{lm:compare-gradient-solution}, we obtain:
\begin{corollary}
\label{cor:compare-gradient}  Let $\A$ satisfy \eqref{structural-reference-1}--\eqref{structural-reference-3}, $\ba \in \G_{B_3}(\eta)$, and $M_0>0$.
For any $\e>0$, there exists $\delta>0$ depending only on $\e$,  $\Lambda$, $p$, $\eta$,  $n$, $\K$ and $M_0$  such that:  if $ \lambda>0$, $0<\theta\leq 1$,
\begin{equation*}
\fint_{B_3}  d_{\A,\ba}(x) \, dx  \leq \delta,\qquad
 \fint_{B_4}  |\F|^{\frac{p}{p-1}} \, dx  \leq \delta,
\end{equation*}
and $u\in W_{loc}^{1,p}(B_4)$ is a weak solution of \eqref{eq-for-u} satisfying
\begin{equation*}
 \Big(\fint_{B_4}{|u|^p\, dx}\Big)^{\frac1p}\leq \frac{M_0}{\lambda \theta}  \quad\mbox{ and }\quad \fint_{B_4}{|\nabla  u|^p \, dx}\leq 1,
\end{equation*}
and $v\in  W^{1,p}(B_3)$ is a weak solution of
\eqref{eq-for-v}
 with   $\|v\|_{L^\infty(B_3)}\leq \frac{M_0}{\lambda \theta}$,
then 
\begin{equation*}
\int_{B_2}{|\nabla  u - \nabla  v|^p\, dx}\leq \e^p.
\end{equation*}
\end{corollary}
\begin{proof} We will present the proof only for the case $1<p<2$ as the case $p\geq 2$ is simpler. Let $\e>0$ be arbitrary. By Lemma~\ref{lm:bound-solution} and the assumptions, we have
\begin{equation*}
 \int_{B_3}{|\nabla u -\nabla v|^p\, dx}\leq C(p,n,\Lambda)\int_{B_4} \big( |u|^p + |\F|^{\frac{p}{p-1}} \big) \, dx
 \leq   C^* \, \Big[\big(\frac{M_0}{\lambda\theta}\big)^p +\delta\Big].
\end{equation*}
Therefore, the conclusion of the lemma follows if $\lambda \theta > \frac{(2C^*)^{1/p} M_0}{\e}$. Thus, it remains to consider the case
\begin{equation}\label{lambda-theta-restriction}
 \lambda \theta \leq \frac{(2C^*)^{\frac1p} M_0}{\e}.
\end{equation}
Now from Lemma~\ref{lm:compare-gradient-solution}(ii) and the boundedness of $d_{\A,\ba}(x)$ we get
\begin{align}\label{first-conse}
\int_{B_2} |\nabla u -\nabla  v|^p  dx 
&\leq  \frac{C}{ \sigma^{\frac{p}{p-1}}} \left\{\int_{B_{\frac52}}{\omega\big(|\lambda\theta( u-v)|\big)^{\frac{p}{p-1}}  \,  dx}+
\int_{B_{\frac52}}{\big(d_{\A,\ba} + |\F|^{\frac{p}{p-1}}\big) \, dx}\right\}\\
&\quad +C\sigma^{\frac{p}{2-p}}  + \frac{C}{\sigma}\|u-v\|_{L^p(B_{\frac52})}\qquad \mbox{for all small}\quad \sigma>0.\nonumber
\end{align}
Notice that for any $\tau>0$ small, from the definition of $\omega$ in \eqref{modified-Lipschitz} we have
\[
\omega(s)\leq \tau + \frac{2\, \Lambda^p}{\tau^{p-1}}s^{p-1}\quad \forall s\geq 0.
\]
Therefore by combining with \eqref{lambda-theta-restriction}, one easily sees that 
\[
\omega\big(|\lambda\theta( u-v)|\big)^{\frac{p}{p-1}} \leq \tau + \frac{C}{\tau^p} (\lambda\theta)^p|u-v|^p
\leq \tau + \frac{C}{\tau^p\e^p} |u-v|^p\quad \forall \tau>0 \,\mbox{ small}.
\]
Hence  by first selecting $\sigma =\sigma(\epsilon, p)>0$ small such that $C\sigma^{\frac{p}{2-p}} \leq \epsilon^p/5$ and then choosing $\tau>0$ such that $C\tau /\sigma^{\frac{p}{p-1}}\leq \epsilon^p/5$, we conclude from \eqref{first-conse} that
 \begin{align}\label{second-conse}
\int_{B_2} |\nabla u -\nabla  v|^p  dx 
\leq  2\frac{\epsilon^p}{5}
+ \frac{C }{ \tau^p\e^p \sigma^{\frac{p}{p-1}}}\|u-v\|_{L^p(B_{\frac52})}^p+ \frac{C\delta }{ \sigma^{\frac{p}{p-1}}}
 + \frac{C}{\sigma}\|u-v\|_{L^p(B_{\frac52})}.
\end{align}
Next let us pick $\epsilon'>0$ small so that
\[
\frac{C }{ \tau^p\e^p \sigma^{\frac{p}{p-1}}} (\epsilon')^p\leq \frac{\epsilon^p}{5}\quad \mbox{and}\quad  \frac{C}{\sigma} \epsilon'\leq \frac{\epsilon^p}{5}.
\]
Then by Lemma~\ref{lm:compare-solution} there exists $\delta'>0$ such that $\|u-v\|_{L^p(B_3)}\leq \epsilon'$ if $\fint_{B_3}d_{\A,\ba}(x)\, dx \leq \delta'$ and  $\fint_{B_4}|\F|^{\frac{p}{p-1}} \, dx\leq \delta'$. Thus, by taking $\delta :=\min{\big\{\delta', \frac{\epsilon^p \sigma^{\frac{p}{p-1}}}{5C}\big\}}$ we obtain the  corollary from \eqref{second-conse}.
 \end{proof}

The next result is a localized version of Corollary~\ref{cor:compare-gradient}.

\begin{lemma}
\label{lm:localized-compare-gradient}  Let $\A$ satisfy \eqref{structural-reference-1}--\eqref{structural-reference-3}, and $M_0>0$.
For any $\e>0$, there exists $\delta>0$ depending only on $\e$,  $\Lambda$, $p$,  $\eta$,  $n$, $\K$ and $M_0$ such that:  if $ \lambda>0$, $0<\theta\leq 1$, $0<r\leq 1$,
\begin{equation}\label{localized-AAcCond}
\dist\Big(\A, \G_{B_{3 r}}\Big) < \delta \quad \mbox{and}\quad \fint_{B_{4r}}  |\F|^{\frac{p}{p-1}} \, dx  \leq \delta,
\end{equation}
and $u\in W_{loc}^{1,p}(B_{4 r})$ is a weak solution of $\div \Big[\frac{\A(x,\lambda\theta u, \lambda \nabla u)}{\lambda^{p-1}}\Big] = \div \F\,$ in $B_{4 r}$ satisfying
\begin{equation*}
 \|u\|_{L^\infty(B_{4r})}\leq \frac{M_0}{\lambda\theta} \quad\mbox{ and }\quad \fint_{B_{4 r}}{|\nabla  u|^p \, dx}\leq 1,
\end{equation*}
then 
\begin{equation}\label{gradients-close}
\fint_{B_{2 r}}{|\nabla  u - \nabla  v|^p\, dx}\leq \e^p
\end{equation}
for some function $v\in  W^{1,p}(B_{3 r})$  with
\begin{equation}\label{property-for-v}
\|\nabla  v\|_{L^\infty(B_{\frac{3 r}{2}})}^p
\leq C(p,n,\eta, \Lambda, \K, M_0) \fint_{B_{2r}}{|\nabla  v|^p \, dx}.
\end{equation}
\end{lemma}
\begin{proof}
It follows from Definition~\ref{def:rescaled-class} and the first condition in \eqref{localized-AAcCond} that there exists a vector field $\ba'\in \G_{B_3}(\eta)$ such that
\begin{equation*}
 \fint_{B_{3 r}} d_{\A,\ba}(y) \,dy\leq \delta\quad \mbox{with}\quad \ba(y,z,\xi) := \ba'\big(\frac{y}{r},z,\xi\big). 
\end{equation*}
Define
\[ \A'(x, z, \xi) = \A(r x, z, \xi),
\quad \F'(x) = \F(r x),  \quad u'(x) =\frac{u(r x)}{r}.
\]
Let   $\theta' := \theta r\in (0,1]$. Then $u'\in W^{1,p}_{loc}(B_4)$ is a weak solution of
\begin{equation*}
\div \Big[\frac{\A'(x,\lambda\theta' u',\lambda \nabla u')}{\lambda^{p-1}}\Big]=\div \F' \quad\mbox{in}\quad B_4.
\end{equation*}
Notice that  $\|u'\|_{L^\infty(B_4)}\leq \frac{M_0}{\lambda \theta'}$ and
$
d_{\A',\ba'}(x) =d_{\A,\ba}(rx)$.
Thus we also have
\begin{align*}
&\Big(\fint_{B_4}{|u'(x)|^p\, dx}\Big)^{\frac1p}= \frac{1}{r}\Big(\fint_{B_{4 r}} |u(y)|^p\, dy\Big)^\frac1p \leq \frac{M_0}{\lambda\theta'},\qquad \fint_{B_{4}}{|\nabla  u'(x)|^p \, dx}= \fint_{B_{4 r}}{|\nabla  u(y)|^p \, dy}\leq 1,\\
&\fint_{B_{3}}  d_{\A',\ba'}(x)\,dx= \fint_{B_{3r}}  d_{\A,\ba}(y)\,dy, \qquad \fint_{B_{4}}  |\F'(x)|^{\frac{p}{p-1}} \, dx = \fint_{B_{4 r}}  |\F(y)|^{\frac{p}{p-1}} \, dy.
\end{align*}
Therefore, given any $\e>0$, by Corollary~\ref{cor:compare-gradient} there exists a constant $\delta=\delta(\e,\Lambda, p, \eta,n, \K, M_0)>0$ such that if 
condition \eqref{localized-AAcCond} for $\A$ and  $\F$ is satisfied then we have
\begin{equation}\label{gradient-compare-universal-ball}
\int_{B_{2}}{|\nabla  u'(x) - \nabla  v'(x)|^p\, dx}\leq  2^n \omega_n\e^p,
\end{equation}
where $v'\in  W^{1,p}(B_{3})$ is a weak solution of
\begin{equation*}
\left \{
\begin{array}{lcll}
\div \Big[\frac{\ba'(x,\lambda\theta' v',\lambda \nabla v')}{\lambda^{p-1}}\Big] &=&  0 \quad &\text{in}\quad B_{3}, \\
v' & =& u'\quad &\text{on}\quad \partial B_{3}
\end{array}\right.
\end{equation*}
satisfying  $\|v'\|_{L^\infty(B_3)}\leq \frac{M_0}{\lambda \theta'}$.
Notice that the existence of such weak solution $v'$ to the above  Dirichlet problem  is guaranteed  by Remark~\ref{existen
ce+uniqueness}.  Now let $v(x) := r v'(x/r)$ for $x\in B_{3 r}$. Then  
by changing variables, we obtain the desired estimate \eqref{gradients-close} from \eqref{gradient-compare-universal-ball}. 

It remains to show \eqref{property-for-v}. Define $\bar{v}(y) = \lambda \theta' v'(y/\theta')$.  Then  $\|\bar v\|_{L^\infty(B_{3 \theta'})}\leq M_0$ and  $\bar{v}$ is a weak solution of
\begin{equation*}
\div \ba'(y,\bar v,\nabla \bar v) =  0 \quad\mbox{in}\quad B_{3 \theta'}.
\end{equation*}
Since $0<\theta'\leq 1$ and $\ba'\in \G_{B_3}( \eta)$, the  assumption ({\bf H3}) about  interior $W^{1,\infty}$-estimates gives
\begin{align*}
\|\nabla  \bar v\|_{L^\infty(B_{\frac{3\theta'}{2}})}^p\leq C(p,n,\eta, \Lambda, \K, M_0) \fint_{B_{2\theta'}}{|\nabla \bar{v}|^p \, dx}.
\end{align*}
This yields \eqref{property-for-v} owing to $\bar v(y) = \lambda\theta  v(y/\theta)$ and $\theta'/ \theta=r$.

\end{proof}

\begin{remark}\label{rm:translation-invariant}
Since the class of our equations is invariant under the transformation $x\mapsto x+y$, Lemma~\ref{lm:localized-compare-gradient} still holds true if $B_r$ is replaced by $B_r(y)$.
\end{remark}

\section{Density and gradient estimates}\label{interior-density-gradient}
We will derive interior $W^{1,q}$-estimates for solution $u$ of \eqref{ME} by estimating the distribution functions of the maximal function of $|\nabla  u|^p$. The precise maximal operators will be used are:
\begin{definition} The Hardy--Littlewood maximal function of a  function 
$f\in L^1_{loc}(\R^n)$ is defined by
\[
 (\M f)(x) = \sup_{\rho>0}\fint_{B_\rho(x)}{|f(y)|\, dy}.
\] 
In case $U$ is a region in $\R^n$ and $f\in L^1(U)$, then we   denote
$ \M_U f = \M (\chi_U f)$.
 \end{definition}
  
The next result gives a density estimate for the distribution  of $\M_{B_5}(|\nabla  u|^p)$. It roughly says that  if the maximal function $\M_{B_5}(|\nabla  u|^p)$ is bounded at one point in $B_r(y)$ then this property can be propagated for all points in $B_r(y)$ except on a set of small measure. 
\begin{lemma}\label{initial-density-est}
Assume that   $\A$ satisfies \eqref{structural-reference-1}--\eqref{structural-reference-3},  $\F\in L^{\frac{p}{p-1}}(B_6;\R^n)$, and $M_0>0$.  There exists a constant $N>1$ depending only on $p$, $n$, $\eta$,  $\Lambda$, $\K$ and $M_0$  such that   
for  any $\e>0$, we can find   $\delta=\delta(\e,\Lambda, p, \eta,n, \K, M_0)>0$  satisfying:  if $\lambda>0$, $0<\theta\leq 1$,
\begin{equation*}\label{interior-SMO}
\sup_{0<\rho\leq 3}\sup_{y\in B_1} \dist \Big(\A, \G_{B_\rho(y)}(\eta)\Big)< \delta,
\end{equation*}
then
  for any weak solution $u$ of \eqref{ME} with
 $\|u\|_{L^\infty(B_5)}\leq \frac{M_0}{\lambda \theta}$,
and for any $y\in B_1$, $0<r\leq 1$ with
\begin{equation}\label{one-point-condition}
 B_r(y)\cap B_1\cap \big\{B_5:\, \M_{B_5}(|\nabla  u|^p)\leq 1 \big\}\cap \{ B_5: \M_{B_5}(|\F|^{\frac{p}{p-1}} )\leq \delta\}\neq \emptyset,
\end{equation}
we have
\begin{equation*}
 \big| \{ B_1:\, \M_{B_5}(|\nabla  u|^p)> N \}\cap B_r(y)\big|
\leq \e  |B_r(y)|.
\end{equation*}
\end{lemma}
\begin{proof}
By condition \eqref{one-point-condition}, there exists a point $x_0\in B_r(y)\cap  B_1$ such that
\begin{align}
  \M_{B_5}(|\nabla  u|^p)(x_0)\leq 1\quad \mbox{and}\quad \M_{B_5}(|\F|^{\frac{p}{p-1}})(x_0) \leq
\delta.\label{maximal-fns-control}
\end{align}
Since $B_{4 r}(y) \subset B_{5 r}(x_0) \cap B_5$, it follows from \eqref{maximal-fns-control} that
\begin{align*}
&\fint_{B_{4 r}(y)}|\nabla  u|^p \, dx  \leq \frac{|B_{5 r}(x_0)|}{|B_{4 r}(y)|} \frac{1}{|B_{5 r}(x_0)|}\int_{B_{5 r}(x_0)\cap B_5}|\nabla  u|^p \, dx \leq \Big(\frac{5}{4}\Big)^{n},\\
 &\fint_{B_{4 r}(y)} |\F|^{\frac{p}{p-1}} \, dx  \leq \frac{|B_{5 r}(x_0)|}{|B_{4 r}(y)|} \frac{1}{|B_{5 r}(x_0)|}\int_{B_{5 r}(x_0)\cap B_5} |\F|^{\frac{p}{p-1}} \, dx \leq \Big(\frac{5}{4}\Big)^{n}\delta.
\end{align*}
Therefore, we can use Lemma~\ref{lm:localized-compare-gradient} and Remark~\ref{rm:translation-invariant} to obtain
\begin{equation}\label{gradient-comparison}
\fint_{B_{2r}(y)}{|\nabla  u- \nabla  v|^p\, dx }\leq \gamma^p,
\end{equation}
where $v \in W^{1,p}(B_{3r}(y))$ is some function satisfying
\begin{equation}\label{translated-v-property}
\|\nabla  v\|_{L^\infty(B_{\frac{3 r}{2}}(y))}^p
\leq C(p,n,\eta,\Lambda,\K, M_0 ) \fint_{B_{2r}(y)}{|\nabla  v|^p \, dx}.
\end{equation}
Here  $\delta=\delta(\gamma,\Lambda, p, \eta,n, \K, M_0)>0$   with $\gamma\in (0,1)$ being determined later. By using \eqref{translated-v-property}
 together with \eqref{gradient-comparison} and \eqref{maximal-fns-control}, we get
\begin{align}\label{gradient-is-bounded}
\|\nabla  v\|_{L^\infty(B_{\frac{3 r}{2}}(y))}^p
\leq  2^{p-1} C \left( \fint_{B_{2r}(y)}{|\nabla  u -\nabla  v|^p \, dx}
+ \fint_{B_{2r}(y)}{|\nabla  u|^p \, dx}\right)
\leq C_* (\gamma^p +1),
\end{align}
where $C_*=C_*(p,n,\eta,\Lambda,\K, M_0)$.
We claim that \eqref{maximal-fns-control}, \eqref{gradient-comparison} and \eqref{gradient-is-bounded}
yield
\begin{equation}\label{set-relation-claim}
 \big\{ B_r(y):  \M_{B_{2r}(y)}(|\nabla  u - \nabla  v|^p) \leq C_* \big\}\subset \big\{ B_r(y):\, \M_{B_5}(|\nabla  u |^p) \leq N\big\}
\end{equation}
with $N := \max{\{2^{p+1} C_*, 5^{n}\}}$. Indeed, let $x$ be a point in the set on the left hand side of \eqref{set-relation-claim}, and consider 
$B_\rho(x)$. If $\rho\leq r/2$, then $B_\rho(x) \subset B_{3 r/2}(y)\subset B_3$
and hence
\begin{align*}
 \frac{1}{|B_\rho(x)|}\int_{B_\rho(x) \cap B_5} |\nabla  u|^p \, dx
 &\leq 
 \frac{2^{p-1}}{|B_\rho(x)|}\Big[\int_{B_\rho(x) \cap B_5} |\nabla  u-\nabla  v|^p \, dx
 +\int_{B_\rho(x) \cap B_5} |\nabla  v|^p\, dx \Big]\\
 &\leq 2^{p-1} \Big[\M_{B_{2r}(y)}(|\nabla  u - \nabla  v|^p)(x) + \|\nabla  v \|_{L^\infty(B_{\frac{3r}{2}}(y))}^p\Big]\\
 &\leq 2^{p-1} C_* \big( \gamma^p + 2\big)\leq 2^{p+1} C_*.
\end{align*}
On the other hand if $\rho>r/2$, then  $B_\rho(x)\subset B_{5\rho}(x_0)$. This and  the first inequality in \eqref{maximal-fns-control} imply that
\begin{align*}
 \frac{1}{|B_\rho(x)|}\int_{B_\rho(x) \cap B_5} |\nabla  u|^p \, dx
 \leq \frac{5^{n}}{|B_{5\rho}(x_0)|} \int_{B_{5\rho}(x_0) \cap B_5} |\nabla  u|^p \, dx\leq 
 5^{n}.
\end{align*}
Therefore, $\M_{B_5}(|\nabla  u |^p)(x) \leq N$ and the claim \eqref{set-relation-claim} is proved.
Note that   \eqref{set-relation-claim} is equivalent to
 \begin{align*}
 \big\{B_r(y):\, \M_{B_5}(|\nabla  u |^p) > N\big\} \subset
 \big\{ B_r(y):\, \M_{B_{2r}(y)}(|\nabla  u - \nabla  v|^p) >C_* \big\}.
 \end{align*}
It follows from this,  the weak type $1-1$ estimate and \eqref{gradient-comparison}   that
\begin{align*}
 &\big|\big\{B_r(y):\, \M_{B_5}(|\nabla  u |^p) > N\big\} \big|\leq  \big|
 \big\{ B_r(y):\, \M_{B_{2r}(y)}(|\nabla  u - \nabla  v|^p) > C_* \big\}\big|\\
 &\leq C  \int_{B_{2r}(y)}{|\nabla  u - \nabla  v|^p \, dx }\leq C' \gamma^p \, |B_r(y)|,
 \end{align*}
 where $C'>0$ depends only on $p$, $n$, $\eta$, $\Lambda$, $\K$ and $M_0$.
By choosing $\gamma = \sqrt[p]{\frac{\e}{C'}}$, we obtain the desired result.
\end{proof}

In view of Lemma~\ref{initial-density-est}, we can apply the variation of the Vitali covering lemma given by \cite[Theorem~3]{W} (see also \cite[Lemma~1.2]{CP}) 
for \[
C=\{ B_1:\, \M_{B_5}(|\nabla   u|^p)> N \} \quad\mbox{and}\quad  
D=\{ B_1:\, \M_{B_5}(|\nabla  u|^p)> 1 \}\cup \{ B_1:\, \M_{B_5}(|\F|^{\frac{p}{p-1}})> \delta \}
\]
 to obtain:
\begin{lemma}\label{second-density-est}
Assume that   $\A$ satisfies \eqref{structural-reference-1}--\eqref{structural-reference-3}, $\F\in L^{\frac{p}{p-1}}(B_6;\R^n)$, and $M_0>0$. There exists a constant $N>1$ depending only on $p$, $n$, $\eta$, $\Lambda$, $\K$ and $M_0$  such that   
for  any $\e>0$, we can find   $\delta=\delta(\e,\Lambda, p, \eta, n, \K, M_0)>0$  satisfying:  if $\lambda>0$, $0<\theta\leq 1$,
\begin{equation*}
\sup_{0<\rho\leq 3}\sup_{y\in B_1} \dist \Big(\A, \G_{B_\rho(y)}(\eta)\Big)< \delta,
\end{equation*}
then   for any weak solution $u\in W_{loc}^{1,p}(B_6)$ of \eqref{ME} satisfying
\begin{equation*}
 \|u\|_{L^\infty(B_5)}\leq \frac{M_0}{\lambda \theta} \quad \mbox{ and }\quad
\big| \{B_1: 
\M_{B_5}(|\nabla  u|^p)> N \}\big| \leq \e |B_1|,
\end{equation*}
we have
\begin{align*}
\big|\{B_1: \M_{B_5}(|\nabla u|^p)> N\}\big|
\leq 20^n \e \, \Big(
\big|\{B_1: \M_{B_5}(|\nabla u|^p)> 1\}\big|
+ \big|\{ B_1: \M_{B_5}(|\F|^{\frac{p}{p-1}})> \delta \}\big|\Big).\nonumber
\end{align*}
\end{lemma}

\subsection{Interior gradient estimates in Lebesgue spaces}\label{sec:Lebesgue-Spaces}

We are now ready to prove  Theorem~\ref{main-result}.

\begin{proof}[\textbf{Proof of Theorem~\ref{main-result}}]
Let  $N>1$ be as in  Lemma~\ref{second-density-est}, and let $q_1=q/p>1$. We choose  $\e=\e(p, q, n,\eta,\Lambda, \K, M_0 )>0$ be such that
\[
\e_1 \eqdef 20^n \e = \frac{1}{2 N^{q_1}}, 
\]
and let $\delta=\delta(p,q,n ,\Lambda, \eta, \K, M_0)$  be the corresponding constant given by Lemma~\ref{second-density-est}.

Assuming for a moment that $u$  satisfies 
\begin{equation}\label{initial-distribution-condition}
\big| \{B_1: 
\M_{B_5}(|\nabla u|^p)> N \}\big| \leq \e |B_1|.
\end{equation}
Then it follows from  Lemma~\ref{second-density-est} that
\beq\label{initial-distribution-est}
\big|\{B_1: \M_{B_5}(|\nabla u|^p)> N\}\big|
\leq \e_1  \left(
\big|\{B_1: \M_{B_5}(|\nabla u|^p)> 1\}\big|
+ \big|\{ B_1: \M_{B_5}(|\F |^{\frac{p}{p-1}})> \delta \}\big|\right).
\eeq
 Let us iterate this estimate by considering
\[
u_1(x) = \frac{u(x)}{N^{\frac1p}}, \quad \F_1(x) = \frac{\F(x)}{N^{\frac{p-1}{p}}}\quad \mbox{and}\quad \lambda_1 = N^{\frac1p}\lambda.
\]
It is clear that $\|u_1\|_{L^\infty(B_5)}\leq \frac{M_0}{\lambda_1 \theta}$ and  $u_1\in W_{loc}^{1,p}(B_6)$    is a weak solution of
\begin{equation*}
\div \Big[\frac{\A(x,\lambda_1 \theta u_1, \lambda_1 \nabla u_1)}{\lambda_1^{p-1}}  \Big]=\div  \F_1  \quad\mbox{in}\quad B_6.
\end{equation*}
Moreover, thanks to \eqref{initial-distribution-condition} we have
\begin{align*}
\big| \{B_1: 
\M_{B_5}(|\nabla u_1|^p)> N \}\big| &= \big| \{B_1: 
\M_{B_5}(|\nabla u|^p)> N^2 \}\big| \leq \e |B_1|.
\end{align*}
Therefore, by applying  Lemma~\ref{second-density-est} to $u_1$ we obtain
\begin{align*}
\big|\{B_1: \M_{B_5}(|\nabla u_1|^p)> N\}\big|
&\leq \e_1 \left( 
\big|\{B_1: \M_{B_5}(|\nabla u_1|^p)> 1 \}\big|
+ \big|\{ B_1: \M_{B_5}(|\F_1 |^{\frac{p}{p-1}} )> \delta \}\big|\right)\\
&= \e_1  \left( 
\big|\{B_1: \M_{B_5}(|\nabla u|^p)> N \}\big|
+ \big|\{ B_1: \M_{B_5}(|\F|^{\frac{p}{p-1}})> \delta N\}\big| \right).
\end{align*}
We infer from this and  \eqref{initial-distribution-est} that
\begin{align}\label{first-iteration-est}
&\big|\{B_1: \M_{B_5}(|\nabla u|^p)> N^2\}\big|
\leq \e_1^2 
\big|\{B_1: \M_{B_5}(|\nabla u|^p)> 1 \}\big|\\
&\qquad + \e_1^2\big|\{ B_1: \M_{B_5}(|\F|^{\frac{p}{p-1}})> \delta\} \big|+ \e_1\big|\{ B_1: \M_{B_5}(|\F|^{\frac{p}{p-1}})> \delta N\}\big|. \nonumber
\end{align}
 Next, let
\[
u_2(x) = \frac{u(x)}{N^{\frac2p}}, \quad \F_2(x) = \frac{\F(x)}{N^{\frac{2(p-1)}{p}}} \quad \mbox{and}\quad \lambda_2 = N^{\frac2p}\lambda.
\]
Then  $u_2$  is a weak solution of
\begin{equation*}
\div \Big[\frac{\A(x,\lambda_2 \theta u_2, \lambda_2 \nabla u_2)}{\lambda_2^{p-1}}  \Big]=\div \F_2 \quad\mbox{in}\quad B_6
\end{equation*}
satisfying
\begin{align*}
  \|u_2\|_{L^\infty(B_5)}\leq \frac{M_0}{\lambda_2 \theta}\quad \mbox{and}\quad  
\big| \{B_1: 
\M_{B_5}(|\nabla u_2|^p)> N \}\big| &= \big| \{B_1: 
\M_{B_5}(|\nabla u|^p)> N^3 \}\big| \leq \e  |B_1|.
\end{align*}
Hence by applying  Lemma~\ref{second-density-est} to $u_2$ we get
\begin{align*}
\big|\{B_1: \M_{B_5}(|\nabla u_2|^p)> N\}\big|
&\leq \e_1 \left( 
\big|\{B_1: \M_{B_5}(|\nabla u_2|^p)> 1 \}\big|
+ \big|\{ B_1: \M_{B_5}(|\F_2|^{\frac{p}{p-1}} )> \delta\}\big|\right)\\
&= \e_1  \left( 
\big|\{B_1: \M_{B_5}(|\nabla u|^p)> N^2 \}\big|
+ \big|\{ B_1: \M_{B_5}(|\F|^{\frac{p}{p-1}})> \delta N^2\}\big| \right).
\end{align*}
This together with  \eqref{first-iteration-est} gives
\begin{align*}
\big|\{B_1: \M_{B_5}(|\nabla u|^p)> N^3\}\big|
\leq \e_1^3 
\big|\{B_1: \M_{B_5}(|\nabla u|^p)> 1 \}\big|
+ \sum_{i=1}^3\e_1^i\big|\{ B_1: \M_{B_5}(|\F|^{\frac{p}{p-1}})> \delta N^{3-i}\} \big|.
\end{align*}
By repeating the iteration, we then conclude that
\begin{align}\label{decay-distr}
\big|\{B_1: \M_{B_5}(|\nabla u|^p)> N^k\}\big|
&\leq \e_1^k 
\big|\{B_1: \M_{B_5}(|\nabla u|^p)> 1 \}\big|
+ \sum_{i=1}^k\e_1^i\big|\{ B_1: \M_{B_5}(|\F|^{\frac{p}{p-1}})> \delta N^{k-i}\} \big|
\end{align}
for all $k=1,2,\dots$ This together with  
\begin{align*}
&\int_{B_1}\M_{B_5}(|\nabla u|^p)^{q_1} \, dx  =q_1 \int_0^\infty t^{q_1-1} \big|\{B_1: \M_{B_5}(|\nabla u|^p)>t\}\big|\, dt\\
&=q_1 \int_0^{N} t^{q_1 -1} \big|\{B_1: \M_{B_5}(|\nabla u|^p)>t\}\big|\, dt
+q_1 \sum_{k=1}^\infty\int_{N^{k}}^{N^{k+1}} t^{q_1 -1} \big|\{B_1: \M_{B_5}(|\nabla u|^p)>t\}\big|\, dt\\
&\leq N^{q_1} |B_1| + (N^{q_1} -1) \sum_{k=1}^\infty N^{q_1 k}\big|\{B_1: \M_{B_5}(|\nabla u|^p)>N^k\}\big|
\end{align*}
gives
\begin{align*}
\int_{B_1}\M_{B_5}(|\nabla u|^p)^{q_1} \, dx  
&\leq  N^{q_1} |B_1| + (N^{q_1} -1)|B_1| \sum_{k=1}^\infty (\e_1 N^{q_1})^k\\
&\quad +\sum_{k=1}^\infty\sum_{i=1}^k (N^{q_1} -1)N^{q_1 k}\e_1^i\big|\{ B_1: \M_{B_5}(|\F|^{\frac{p}{p-1}})> \delta N^{k-i}\} \big|.
\end{align*}
But we have
\begin{align*}
&\sum_{k=1}^\infty\sum_{i=1}^k (N^{q_1} -1)N^{{q_1} k}\e_1^i\big|\{ B_1: \M_{B_5}(|\F|^{\frac{p}{p-1}})> \delta N^{k-i}\} \big|\\
&=\big(\frac{N}{\delta}\big)^{q_1}\sum_{i=1}^\infty (\e_1 N^{q_1})^i\left[\sum_{k=i}^\infty (N^{q_1} -1)\delta^{{q_1}}N^{{q_1} (k-i-1)}\big|\{ B_1: \M_{B_5}(|\F|^{\frac{p}{p-1}})>
\delta N^{k-i}\} \big|\right]\\
&=\big(\frac{N}{\delta}\big)^{q_1}\sum_{i=1}^\infty (\e_1 N^{q_1})^i\left[\sum_{j=0}^\infty (N^{q_1} -1)\delta^{{q_1}}N^{{q_1} (j-1)}\big|\{ B_1: \M_{B_5}(|\F|^{\frac{p}{p-1}})> \delta N^{j}\} \big|\right]\\
&\leq \big(\frac{N}{\delta}\big)^{q_1} \Big[\int_{B_1}\M_{B_5}(|\F|^{\frac{p}{p-1}})^{q_1} \, dx \Big]\sum_{i=1}^\infty (\e_1 N^{q_1})^i.
\end{align*}
Thus we infer that
\begin{align*}
\int_{B_1}\M_{B_5}(|\nabla u|^p)^{q_1} \, dx 
&\leq N^{q_1} |B_1| + \left[ (N^{q_1} -1)|B_1| +\big(\frac{N}{\delta}\big)^{q_1} \int_{B_1}\M_{B_5}(|\F|^{\frac{p}{p-1}})^{q_1} \, dx \right] \sum_{k=1}^\infty (\e_1 N^{q_1})^k\\
&= N^{q_1} |B_1| + \left[ (N^{q_1} -1)|B_1| +\big(\frac{N}{\delta}\big)^{q_1} \int_{B_1}\M_{B_5}(|\F|^{\frac{p}{p-1}})^{q_1} \, dx \right] \sum_{k=1}^\infty 2^{-k}\\
&\leq C\left( 1+ \int_{B_1}\M_{B_5}(|\F|^{\frac{p}{p-1}})^{q_1} \, dx  \right)
\end{align*}
with the constant $C$ depending only on $p$, $q$, $n$, $ \Lambda$,  $\eta$,  $\K$ and $M_0$. On the other hand, $
|\nabla u(x)|^p  \leq  
\M_{B_5}(|\nabla u|^p)(x)$
for almost every $x\in B_1$. Therefore,
it follows from the strong type ${q_1}-{q_1}$ estimate for the maximal function and the fact ${q_1}=q/p$ that 
\begin{equation}\label{initial-desired-L^p-estimate}
\int_{B_1}|\nabla u|^q \, dx \leq C\left( 1+ \int_{B_5}|\F|^{\frac{q}{p-1}} \, dx \right).
\end{equation}
We next remove the extra assumption  
\eqref{initial-distribution-condition} for $u$. Notice that for any $M>0$, by using the weak type $1-1$ estimate for the maximal function and Lemma~\ref{W^{1,p}-est}  we get
\begin{align*}
\big| \{B_1: 
\M_{B_5}(|\nabla u|^p)> N  M^p\}\big|
\leq \frac{C }{N M^p}\int_{B_5} |\nabla u|^p \, dx \leq \frac{C(p,n,\eta,\Lambda) }{M^p}\, \Big(\int_{B_6} |u|^p \, dx +\int_{B_6} |\F|^{\frac{p}{p-1}} \, dx \Big).
\end{align*}
Therefore, if we let
\[\bar{u}(x,t) = \frac{u(x,t)}{M}\quad \mbox{with}\quad M^p= \frac{C(p,n,\eta,\Lambda)\Big[  \|u\|_{L^p(B_6)}^p +\||\F|^{\frac{1}{p-1}}\|_{L^p(B_6)}^p \Big]}{ \e |B_1|}
\]
then 
$\big| \{B_1: 
\M_{B_5}(|\nabla \bar{u}|^p)> N\}\big|
\leq \e |B_1|$. Hence we can apply \eqref{initial-desired-L^p-estimate} to $\bar{u}$ with $\F$ and $\lambda$ being replaced by 
$\bar \F= \F/M^{p-1}$ and $\bar \lambda = \lambda M$. By reversing back to the  functions $u$ and  $\F$, we obtain the desired estimate  \eqref{main-estimate}.
\end{proof}

We next show that Theorem~\ref{simplest-main-result} is just a special case of Theorem~\ref{main-result}.

\begin{proof}[\textbf{Proof of Theorem~\ref{simplest-main-result}}]

For each $0<\rho\leq 3$ and $y\in B_1$, let
\[
a_{\rho,y}(z, \xi) := \fint_{B_3} \A(y + \frac{\rho}{3}x , z, \xi)\, dx=\A_{B_\rho(y)}(z,\xi).
\]
Then it is easy to see from the assumptions for  $\A$  that  $a_{\rho,y}$  satisfies \eqref{structural-reference-1}--\eqref{structural-reference-3} and  {(\bf{H1})}--{(\bf{H2})}  with $\eta:=\eta_0\equiv 0$. Moreover, $a_{\rho,y}$ also satisfies condition  {(\bf{H3})} thanks to \cite[Theorem~1.2]{HNP2}. These imply that
 $a_{\rho, y}\in \G_{B_3}(\eta_0)$.  Thus  $a_{\rho, y}\in \G_{B_\rho(y)}(\eta_0)$,  and  hence
\[
\dist \Big(\A, \G_{B_\rho(y)}(\eta_0)\Big) \leq 
\fint_{B_\rho(y)} \Big[
\sup_{z\in \overline\K}\sup_{\xi\neq 0}\frac{|\A(x,z,\xi) - \A_{B_\rho(y)}(z,\xi)|}{|\xi|^{p-1}}
\Big] \,dx \leq \delta.
\]
Since this holds for every $0<\rho\leq 3$ and $y\in B_1$, Theorem~\ref{simplest-main-result} follows from Theorem~\ref{main-result}.
\end{proof}

\subsection{Interior gradient estimates in Orlicz spaces}\label{subsec:Orlicz}
In this subsection we show that the interior estimates obtained in Theorem~\ref{main-result} still hold true in {\it Orlicz} spaces. This is achieved by the same arguments
as in Subsection~\ref{sec:Lebesgue-Spaces} which illustrates the robustness of our method. This subsection is motivated by \cite[Section 4]{BW}  where Byun and Wang derived similar estimates for the case $\A(x,z,\xi)$ is independent of $z$.

 Let us first recall 
some basic definitions and  properties about {\it Orlicz} spaces (see \cite{KK,RR}).
A function $\phi: [0,\infty) \to [0, \infty)$ is called a {\it Young function} if it is increasing, convex, and
\[
 \phi(0)=0, \quad \lim_{t\to 0^+}\frac{\phi(t)}{t}=0,  \quad \lim_{t\to \infty}\frac{\phi(t)}{t}=\infty.
\]
Given a {\it Young function} $\phi$ and a bounded domain $U\subset \R^n$, the {\it Orlicz} space $L^{\phi}(U)$ is defined to be the linear hull of $K^{\phi}(U)$ where
\[
 K^{\phi}(U) :=\Big\{g: U\to \R \mbox{ measurable } : \int_{U}{\phi(|g|) \, dx}<\infty \Big\}. 
\]
We will need the following well known conditions for $\phi$.
\begin{definition}
 Let $\phi$ be a {\it Young function}. 
 \begin{itemize}
  \item[(i)] $\phi$ is said to satisfy the $\bigtriangleup_2$-condition if there exists a constant $\mu>1$ such that $\phi(2t)\leq \mu \, \phi(t)$ for every $t\geq 0$. 
   \item[(ii)] $\phi$ is said to satisfy the $\bigtriangledown_2$-condition if there exists a constant $a>1$ such that $\phi(t)  \leq \frac{1}{2a} \phi(a t)$ for every $t\geq 0$. 
 \end{itemize}
We will write $\phi\in \bigtriangleup_2 \cap \bigtriangledown_2$ to mean that $\phi$ satisfies both $(i)$ and $(ii)$. Notice that $\phi\in  \bigtriangledown_2$ implies the quasiconvexity of $\phi$ (see \cite[Lemma~1.2.3]{KK}).
\end{definition}
The next elementary lemma gives a characterization of functions in   $L^{\phi}(U)$ in terms of their distribution functions.
\begin{lemma}\label{distri-int}
 Assume $\phi\in \bigtriangleup_2 \cap \bigtriangledown_2$, $U$ is a bounded domain, and $g:U\to \R$ is a nonnegative measurable function. Let $\nu>0$ and $\alpha>1$. Then
 \[
  g\in L^{\phi}(U)\iff S :=\sum_{j=1}^\infty{\phi(\alpha^j) \,\big|\{x\in U:\, g(x)>\nu \,\alpha^j \}\big|} <\infty.
 \]
Moreover, there exists $C=C(\nu, \alpha,\phi)>0$ such that
\[
 \frac{1}{C} S \leq \int_{U}{\phi(g)\, dx}\leq C (|U| + S).
\]
\end{lemma}
\begin{proof}
This follows from the representation  formula
\[\int_{U}{\phi(|g|) \, dx}
=\int_0^\infty{\big|\{x\in U:\, g(x)>\lambda \}\big|\, d\phi(\lambda)}
\]
and the fact $L^{\phi}(U)\equiv K^{\phi}(U)$ when  $\phi\in \bigtriangleup_2$.
\end{proof}

Now we state the version of Theorem~\ref{main-result} for {\it Orlicz} spaces.
\begin{theorem}\label{second-result} Assume that  $\A$ satisfies 
\eqref{structural-reference-1}--\eqref{structural-reference-3}, and $M_0>0$. 
For any $q>p$, there exists a constant $\delta=\delta(p,q,n ,\Lambda, \eta, \K, M_0)>0$  such that: if $\lambda>0$, $0<\theta \leq 1$,
\begin{equation*}
\sup_{0<\rho\leq 3}\sup_{y\in B_1} \dist \Big(\A, \G_{B_\rho(y)}(\eta)\Big) \leq \delta,
\end{equation*}
and $u\in W_{loc}^{1,p}(B_6)$  is a weak solution of \eqref{ME}
satisfying $\|u\|_{L^\infty(B_5)}\leq\frac{M_0}{\lambda \theta}$, then:
\begin{equation}\label{interior-Orlicz}
|\F|^{\frac{p}{p-1}}  \in L^\phi(B_5) \quad\Longrightarrow \quad |\nabla u|^p\in L^\phi(B_1).
\end{equation}
\end{theorem}

In order to prove Theorem~\ref{second-result}, we need one more  lemma concerning about strong type estimates for maximal functions in {\it Orlicz} spaces.
\begin{lemma}[ Theorem~1.2.1 in \cite{KK} ]\label{maximal-fn-est}
Assume $\phi\in \bigtriangleup_2 \cap \bigtriangledown_2$ and $g\in L^\phi(B_5)$. Then $\M_{B_5}(g)\in L^\phi(B_5)$ and 
\[
 \int_{B_5}{\phi\big(\M_{B_5}(|g|) \big)\, dx}\leq C(n,\phi) \int_{B_5}{\phi(|g|)\, dx}.
\]
\end{lemma}

\begin{proof}[\textbf{Proof of Theorem~\ref{second-result}}]
Since the arguments are essentially the same as those given in Subsection~\ref{sec:Lebesgue-Spaces}, we only indicate the main points.

Let  $N>1$ be as in  Lemma~\ref{second-density-est}. As $\phi\in \bigtriangleup_2$, it is easy to see that there exists $\mu>1$ such that
\begin{equation}\label{Delta-conse}
 \phi(Nt)\leq \mu^{n_0} \phi(t)
=:\mu_1 \phi(t)\quad \forall t\geq 0,
 \end{equation}
 where $n_0\in \N$ depends only on $N$. Let us choose  $\e=\e(p, \phi, n,\eta,\Lambda, \K, M_0 )>0$ be such that
\[
\e_1 \eqdef 20^n \e = \frac{1}{2 \mu_1}, 
\]
and let $\delta=\delta(p,\phi,n ,\Lambda, \eta, \K, M_0)$  be the corresponding constant given by Lemma~\ref{second-density-est}. By considering the function $\bar u := u/M$ instead of $u$ as done at the end of the proof
of Theorem~\ref{main-result}, we can assume without loss of generality that condition
\eqref{initial-distribution-condition} is satisfied. Thus, we obtain estimate \eqref{decay-distr} and hence
\begin{align}\label{geometric-decay}
&\sum_{k=1}^\infty \phi(N^{k})\big|\{B_1: \M_{B_5}(|\nabla u|^p)>N^k\}\big|\\
&\leq |B_1| \sum_{k=1}^\infty \phi(N^k) \e_1^k
+\sum_{k=1}^\infty\sum_{i=1}^k \phi(N^{k})\e_1^i\big|\{ B_1: \M_{B_5}(|\F|^{\frac{p}{p-1}})> \delta N^{k-i}\} \big|=: S_1 + S_2.\nonumber
\end{align} 
It follows from  \eqref{Delta-conse} that $\phi(N^{k}) \leq \mu_1^k \phi(1)$ and $\phi(N^{k}) \leq \mu_1^{i-1}\phi(N^{k-i+1})$ for each $i=1,\dots, k$. Consequently,
\begin{equation}\label{S_1}
 S_1 \leq \phi(1)|B_1| \sum_{k=1}^\infty (\mu_1\e_1)^k =\phi(1)|B_1| \sum_{k=1}^\infty 2^{-k}=\phi(1)|B_1|
\end{equation}
and
\begin{align*}
 S_2&\leq\mu_1^{-1} \sum_{i=1}^\infty (\mu_1\e_1 )^i\left[\sum_{k=i}^\infty \phi(N^{k-i+1})\big|\{ B_1: \M_{B_5}(|\F|^{\frac{p}{p-1}})>
\delta N^{k-i}\} \big|\right]\\
&=\mu_1^{-1}\sum_{i=1}^\infty  (\mu_1\e_1 )^i\left[\sum_{j=1}^\infty \phi(N^j)\big|\{ B_1: \M_{B_5}(|\F|^{\frac{p}{p-1}})> \frac{\delta}{N} N^{j}\} \big|\right].
\end{align*}
Hence by using Lemma~\ref{distri-int} and Lemma~\ref{maximal-fn-est}, we obtain
\begin{align*}
 S_2 &\leq C \sum_{i=1}^\infty (\mu_1\e_1 )^i  \int_{B_1}\phi\Big(\M_{B_5}(|\F|^{\frac{p}{p-1}})\Big) \, dx 
\leq  C \int_{B_1}\phi\Big(|\F|^{\frac{p}{p-1}}\Big) \, dx.
\end{align*}
This together with \eqref{geometric-decay} and \eqref{S_1} yields
\begin{align*}
\sum_{k=1}^\infty \phi(N^{k})\big|\{B_1: \M_{B_5}(|\nabla u|^p)>N^k\}\big|
\leq C \Big[ 1+ \int_{B_1}\phi\Big(|\F|^{\frac{p}{p-1}}\Big) \, dx\Big].
\end{align*}
Therefore, we conclude from Lemma~\ref{distri-int} that $ \M_{B_5}(|\nabla u|^p) \in L^\phi(B_1)$ which gives \eqref{interior-Orlicz} as 
$|\nabla u(x)|^p\leq \M_{B_5}(|\nabla u|^p)(x)$ for a.e. $x\in B_1$.
\end{proof}

\textbf{Acknowledgement.} The authors would like to thank Luan Hoang for fruitful
discussions. T. Nguyen gratefully acknowledges 
the support by a grant  from the Simons Foundation (\# 318995).

  \end{document}